\documentclass[11pt,twoside,a4paper,reqno]{amsart}
\usepackage[utf8]{inputenc}
\usepackage{amsmath}
\usepackage{amssymb}
\usepackage{amsthm}
\usepackage[english]{babel}
\usepackage{tikz-cd}
\usepackage{url}
\usepackage{enumerate}
\usepackage[normalem]{ulem}
\newtheorem{theorem}{Theorem}
\newtheorem{corollary}[theorem]{Corollary}
\newtheorem{definition}[theorem]{Definition}
\newtheorem{example}[theorem]{Example}
\newtheorem{lemma}[theorem]{Lemma}
\newtheorem{remark}[theorem]{Remark}
\newtheorem{proposition}[theorem]{Proposition}
\DeclareRobustCommand{\stirling}{\genfrac\{\}{0pt}{}}

\newcommand{\supp}{\text{supp}}
\newcommand{\app}{\text{app}}
\newcommand{\as}[1]{\textit{as } #1 \rightarrow\infty}
\renewcommand{\d}{\, \text{d}}
\newcommand{\f}[1]{\varphi_{#1}}
\newcommand{\s}[2]{(#1_#2)_{#2 \geq 1}}
\newcommand{\R}{\mathbb{R}}
\newcommand{\xRightarrow}[1]{\overset{#1}{\Rightarrow}}
\newcommand{\fceil}[2]{\left \lceil \frac{#1}{#2}  \right \rceil}
\newcommand{\ffloor}[2]{\left \lfloor \frac{#1}{#2}  \right \rfloor}

\usepackage{xcolor}
\newcommand{\DS}{\displaystyle}

\begin{document}
\title{Poisson Limit Theorems for Systems with product structure}
\author{M. Auer}\thanks{The author thanks the University of Maryland for their hospitality during this work. Special thanks go my advisor  D. Dolgopyat for suggesting this problem, and providing helpful comments throughout. Thanks also go to A. Kanigowski, for supplying his knowledge on many of the examples.}
\date{}
\maketitle
\makeatletter{\renewcommand*{\@makefnmark}{}
\makeatother}
\begingroup
\leftskip4em
\rightskip\leftskip
\begin{small}
\paragraph*{Abstract}

We obtain a Poisson Limit for return times to small sets for product systems. Only one factor is required to be hyperbolic while the second factor is only required to satisfy polynomial deviation bounds for 
ergodic sums. In particular, the second factor can be elliptic or parabolic.
As an application of our main result,
several maps  of the form  Anosov map $\times$ another map are shown to satisfy a Poisson Limit Theorem at typical points, some even at all points.\\
The methods can be extended to certain types of skew products, including $T,T^{-1}$-maps of high rank. 

\end{small}
\par
\endgroup

\part{Results}
\section{Introduction}

One of the prominent limit theorems in classical probability theory is the Poisson Limit Theorem
(PLT). Due to the PLT, in a variety of probabilistic models describing waiting times until unlikely events occur
are  well approximated by exponentially distributed variables. It has been a great discovery that many deterministic systems satisfy the same kind of limit theorems for rare events. \\

Limit distributions of waiting times are most classical for mixing Markov chains, there one considers returns to small cylinders, for example see \cite[Theorem A]{pitskel_1991}. As remarked there, this result can be immediately generalized to systems with a Markov partition, the only caveat being that the sets are still cylinders, so geometrically not the most intuitive class. Nonetheless, waiting time limits for returns to small balls can be shown in concrete settings; for example hyperbolic toral automorphisms \cite[Theorem 2.3]{DGS}, Rychlik-maps and unimodal maps \cite[Theorem 3.2 and 4.1]{BSTV}, partially or nonuniformly hyperbolic maps \cite[Theorem 8]{Dhyplim} \cite[Theorem 3.3]{PSgeometric} \cite[Theorem 3.3]{ccgeo}, some intermittent interval maps \cite[Main Theorem]{indfixpointCG}, and many more. It is sometimes interesting to also ask for explicit rates of convergence, this can be shown under strong mixing properties, see \cite[Theorem 2.1]{HSV}, \cite[Theorem 7]{AS}, \cite[Theorem 8]{compoundpoisHV}.\\
We do not make any claims on completeness of the list of reference given above, for a more complete picture see \cite{ExtremeL}.
\\
Some related topics are extreme value laws \cite{ExtremeFFT}, \cite{ExtremeFFS}, spatiotemporal limits \cite{spatiotempPS}, \cite{Z1} or Borel-Cantelli like Lemmas \cite{borelcantelliHLSW}, \cite{bcHNPV}, \cite{loglargegaps}.
\\
Similar questions can be asked for flows as well, this topic
has not been studied as thoroughly as the question for maps. As shown in \cite{flowPY} for suspension this reduces to the study for maps. Moreover Poisson Limit Theorem for flows can be reduced to Poisson Limit Theorem for 
time 1 map with the target being the set of points which visit $B(x, r)$ in
the next unit of time.\\
\\
From the list above, we see that the PLT is often associated with strong mixing properties of the system. In the present work we construct systems that are not even weakly mixing, but nevertheless satisfy the PLT (a precise definition is given beneath). The systems will have a special structure $S=T\times R$, where $T$ is hyperbolic, but $R$ is not.\\
We will develop a machinery to show the PLT for such systems. \\
This will be used to construct systems satisfying the PLT, but otherwise exhibiting properties uncharacteristic of chaotic systems - like non weak mixing, or zero entropy\footnote{This cannot be done with products, since $h(T\times R)=h(T)+h(R)$. We extend our methods to skew-products of a certain form (Theorem \ref{plt:2}.)}. This shows that the PLT is much more common then it was believed before.
In fact, discovering the most flexible conditions for the validity of the PLT is a promising direction of a future research.

\section{Preliminaries}
\begin{definition}
Given a probability-preserving ergodic dynamical system $(X,\mathcal{A},\mu,T)$ and a measurable set $A\in \mathcal{A}$, we will define the \textit{first return time to $A$} as
\begin{equation*}
\f{A}(x)=\min(n\geq 1\;|\; T^n(x)\in A).
\end{equation*}
The \textit{first return map} shall be denoted by $T_A(x)=T^{\f{A}(x)}(x)$, and the \textit{sequence of consecutive return times} by
\begin{equation*}
\Phi_{A}=(\f{A},\f{A}\circ T_A,\f{A}\circ T_A^2,...).
\end{equation*}
\end{definition}

In the following, for some measurable set $A\in \mathcal{A}$ with $\mu(A)>0$, the measure conditional on $A$ shall be given as $\mu_A(B)=\frac{\mu(A\cap B)}{\mu(A)}$, $B\in \mathcal{A}\cap A$.\\
The first important result in the study of $\f{A}$, was Kac's formula, which calculates the expectation as 
\begin{equation*}
\int_A \f{A} \d\mu_A = \frac{1}{\mu(A)}.
\end{equation*}
Hence it is natural to study limits of $\mu(A)\f{A}$ as $\mu(A)\rightarrow 0$. More explicitly let $\s{A}{l}$ be a sequence of rare events, that is each $A_l$ is measurable with $\mu(A_l)\rightarrow 0$, we want to find weak limits of the form
\begin{equation*}
\mu(A_l)\Phi_{A_l}\xRightarrow{\mu} \Phi \;\;\;\as{l},
\end{equation*}
or
\begin{equation*}
\mu(A_l)\Phi_{A_l}\xRightarrow{\mu_{A_l}} \tilde{\Phi} \;\;\;\as{l},
\end{equation*}
where $\Rightarrow$ denotes convergence in distribution. In the above situation we shall call $\Phi$ the \textit{hitting time limit} and $\tilde{\Phi}$ the \textit{return time limit}. An important fact is that the hitting and return time limits are intimately related, this relation was first formulated in \cite[Main Theorem]{HLV} (albeit only for the first marginal). The analogous relation for the entire process is shown in \cite[Theorem 3.1]{Z2}. For exponential returns, which is what we are concerned with, the result is as follows.

\begin{theorem}\label{hitretrelcor}
Let $(X,\mathcal{A},\mu,T)$ be an ergodic probability preserving dynamical system and let $\s{A}{l}$ be a sequence of rare events. Then $\Phi\overset{d}{=}\Phi_{Exp}$ is equivalent to $\tilde{\Phi}\overset{d}{=}\Phi_{Exp}$, where $\Phi_{Exp}$ is an iid process of standard exponentially distributed random variables.
\end{theorem}

This suggests that we should expect exponential hitting and return time limits for geometrically sensible sequences of rare events. \\
In the following let $X$ be a $C^r$ Riemannian manifold with $\dim(X)=d$ and assume $\mu\ll m_X$ the volume on $X$, with continuous density, say $\frac{\d\mu}{\d m_X}=\rho$. Most of the statements can be reformulated to hold for arbitrary invariant $\mu$, but for the sake of simplicity we shall keep this assumption.
\begin{definition}
\begin{enumerate}[(i)]
\item Let $x^*\in X$ and, for $r>0$, denote by $B_r(x^*)$ the geodesic ball of radius $r$ centred at $x^*$. We will say that \textit{$T$ satisfies the PLT at $x^*$} if
\begin{equation*}
\mu(B_r(x^*))\Phi_{B_r(x^*)}\xRightarrow{\mu} \Phi_{Exp} \;\;\;\textit{ as }r\rightarrow 0.
\end{equation*}
\item Let
\begin{equation*}
PLT:=\{x^*\in X \;|\; \textit{$T$ satisfies the PLT at $x^*$}\}.
\end{equation*}
If $\mu(PLT)=1$ we will say that \textit{$T$ satisfies the PLT almost everywhere}, and if $PLT=X$ we will say that \textit{$T$ satisfies the PLT everywhere}.
\end{enumerate}
\end{definition}
If $T$ is Lipschitz-continuous along the (finite) orbit of a periodic point, then
it \textbf{does not} satisfy the PLT at that point. 
To see this, note first that the PLT at $x^*$ in particular implies, via Theorem \ref{hitretrelcor}, 
that, for each $N\geq 1$,
\begin{equation*}
\mu_{B_r(x^*)}(\f{B_r(x^*)}\leq N)\rightarrow 0 \;\;\;\textit{ as }r\rightarrow 0.
\end{equation*}
Now suppose $x^*$ is a point with period $p$, and say $\rho(x^*)>0$, and $|T^p(x)-T^p(y)|\leq C|x-y|$ near $x^*$ then\footnote{Assuming $C>1$, we have that $B_r(x^*)$ is diffeomorphic via the exponential map $exp_{x^*}$ to a ball in $\R^d$. Wlog assume that $X\subset\R^N$
\begin{align*}
\mu(B_r(x^*))&=\int_{B_r(x^*)} \rho(x) \d m_X(x) =(\rho(x^*)+o(1)) m_X(B_r(x^*))\\
&=(\rho(x^*)+o(1))\int_{(B_r(0))} \sqrt{|\det D_{\textbf{u}} exp_{x^*} (D_{\textbf{u}} exp_{x^*})^t} \d\lambda^d(\textbf{u})\\
&=(\rho(x^*) \sqrt{|\det D_{\textbf{0}} exp_{x^*} (D_{\textbf{0}} exp_{x^*})^t}+o(1)) \lambda^d(B_r(0)),
\end{align*}
where $\lambda^d$ is the $d$-dimensional Lebesgue measure. }
\begin{equation*}
\mu_{B_r(x^*)}(\f{B_r(x^*)}\leq p) \geq \mu_{B_r(x^*)}(B_{\frac{r}{C}}(x^*))= \frac{1}{C^d}+o(1)
\end{equation*}
as $r\rightarrow 0$. \\
A situation involving periodic points can be considered under only slight modification of usual methods - e.g as in \cite[Theorem 3.3 and Theorem 10.1]{Z1} - and one obtains scaled exponentials as limits.\\
\\
The main goal is to prove the PLT (almost) everywhere for some (skew-) product systems.\\
In the following we will consider return times in different systems - namely we will have three different maps $T:X\rightarrow X$ (or $T_y:X\rightarrow X$) which is usually assumed hyperbolic, $R:Y\rightarrow Y$ which is parabolic or elliptic, and $S=T\times R: X\times Y \rightarrow X\times Y$ - in an attempt to keep notation simple we will (by slight abuse of notation) always denote the return times by $\varphi$. Which map is meant will always be clear by the specified set.

\section{The PLT for (skew-)products}

In this paper we study the PLT for systems that can be written as a product (or skew product of a special type). Therefore let $Y$ be another Riemannian $C^{r'}$-manifold with $\dim(Y)=d'$, and assume $R:Y\rightarrow Y$ preserves a probability measure $\nu\ll m_Y$ with continuous density. Instead of $T:X\rightarrow X$, consider now some  $T:X\times Y\rightarrow X$. 
We will prove the PLT for certain systems of the form $S(x,y)=(T(x,y),R(y))$. The case of direct products can be recovered if $T(x,y)=T(x)$ is independent of $y$ (which will be the case for most of our examples). Denote also $T_y(x)=T(x,y)$. We will assume that $T_y$ preserves a probability measure
$\mu$ (independent of $y$). For measurable $A\subset X$ we introduce analogously the \textit{consecutive fiberwise return times} as
\begin{align*}
&\f{A\times Y}(x,y)=\min(j\geq 1\;|\; S^j(x,y)\in A\times Y),\\
&\Phi_{A\times Y}=(\f{A\times Y},\f{A\times Y} \circ S_{A\times Y},\f{A\times Y} \circ S_{A\times Y}^2,\dots),
\end{align*}
where we only fix a small target in the fiber.\\
For our purposes it is convenient to think of $y$ as fixed. For $n\geq 1$ denote $T^n_y(x)=T_{R^{n-1}(y)}(T_{R^{n-2}(y)}(...(T_y(x))))$, and define
\begin{align*}
&\f{A,y}(x)=\f{A,y}^{(1)}(x)=\min(j\geq 1\;|\; T_y^j(x)\in A),\\
&\f{A,y}^{(n+1)}(x)=\min(j\geq 1\;|\; T_y^{\f{A,y}^{(1)}(x)+\f{A,y}^{(2)}(x)+ \dots +\f{A,y}^{(n)}(x)+j}(x)\in A),\\
&\Phi{A,y}=(\f{A,y}^{(1)},\f{A,y}^{(2)},\cdots).
\end{align*}
Clearly the definitions coincide and $\Phi_{A,y}(x)=\Phi_{A\times Y}(x,y)$.s\\
\\
We will list here the main assumptions\footnote{We often only assume a subset of these, most commonly (MEM), (EE) and (BR$(x^*,y^*)$). But we will always state the current assumptions.} we make in order to prove the 
PLT.

\begin{itemize}
\item[(MEM)] For almost all $y\in Y$ there are constants $r>0$, $C>1$ and $\gamma>0$ such that
\begin{equation}\label{ballretasm:mem}
\begin{aligned}
\left|\int_X \right. & \left. \prod_{j=0}^{n-1} f_j \circ T_y^{k_j} \d\mu - \prod_{j=0}^{n-1} \int_X f_j \d\mu\right|\\
&\leq C e^{-\gamma \min_{0\leq j_1 < j_2\leq n-1}|k_{j_1}-k_{j_2}|}\prod_{j=0}^{n-1} ||f_j||_{C^r},
\end{aligned}
\end{equation}
for $n\geq 1$, $f_0,...,f_{n-1}\in C^r$ and $0\leq k_0\leq...\leq k_{n-1}$.
\item[(EE)] There are $r'>0$ and $0<\delta<1$ such that 
\begin{equation}\label{ballretasm:ee}
\left|\left|\sum_{n=0}^{N-1} f\circ R^n -N\int f \d\nu\right|\right|_{L^2(\nu)} \leq C||f||_{C^{r'}} N^{\delta} \;\;\;\forall f\in C^{r'}, \; \forall N\geq 1.
\end{equation}
\item[(LR$(y^*)$)] There is a $c>0$ such that, for $r>0$ and $\nu$-a.e $y\in B_r(y^*)$, we have 
\begin{equation}\label{ballretasm:lr}
\f{B_r(y^*)}(y)\geq c |\log(r)|.
\end{equation}
\item[(SLR$(y^*)$)] There is a $\psi:(0,\infty)\rightarrow (0,\infty)$ with $|\log(r)|=o(\psi(r))$ as $r\rightarrow 0$ such that, for $r>0$ and $\nu$-a.e $y\in B_r(y^*)$, we have 
\begin{equation}\label{ballretasm:slr}
\f{B_r(y^*)}(y)\geq \psi(r).
\end{equation}
\item[(LR'$(x^*)$)] For $\nu$-a.e $y\in Y$ there is a $c=c_y>0$ such that, for $r>0$ and $\mu$-a.e $x\in B_r(x^*)$, we have 
\begin{equation}\label{ballretasm:lr'}
\f{B_r(x^*),y}(x)\geq c |\log(r)|.
\end{equation}
\item[(NSR$(x^*)$)] There is a $\xi:(0,\infty)\rightarrow (0,\infty)$ with $|\log(r)|=o(\xi(r))$ as $r\rightarrow 0$ such that, for $\nu$-a.e $y\in Y$, we have 
\begin{equation}\label{noshortreturns:1}
\mu_{B_r(x^*)}(\f{B_r(x^*),y}\leq \xi(r))\rightarrow 0 \;\;\;\textit{ as }r\rightarrow 0.
\end{equation}
\item[(BR$(x^*,y^*)$)] One of the following is satisfied
\begin{itemize}
\item (SLR$(y^*)$),
\item (NSR$(x^*)$) AND (LR$(y^*)$),
\item or (NSR$(x^*)$) AND (LR'$(x^*)$).
\end{itemize}
\end{itemize}
~\\

Colloquially we will also refer to (MEM) as exponential mixing of all orders, and to (EE) as the Quantitative Ergodic Theorem or effective ergodicity. Both are standard assumptions 
and have been studied for many classes of systems.\\
Conditions (LR), (LR'), (SLR), and (NSR) all are concerned with the fact that a points in a small ball $B$ cannot return to $B$ too quickly. Sometimes in literature the center $x^*$ or $y^*$ is referred to as a slowly recurrent point. For technical reason we need to distinguish different versions of slow recurrence, (SLR) being the strongest.

\begin{remark}\label{conditionsrem}
\begin{enumerate}[(i)]
\item In the case $T(x,y)=G_{\tau(y)}(x)$, where $G$ is a flow satisfying (a continuous version of) (MEM)\footnote{It is in fact enough to assume exponential mixing.} and $\tau$ is bounded,
the condition (NSR$(x^*)$) is satisfied at almost every $x^*$. 
Indeed, it was shown in \cite[Lemma 4.13]{loglargegaps}, albeit for maps instead of flows, that  condition (NSR$(x^*)$) is satisfied\footnote{It is shown that, for every fixed $A,K>0$, we have
\begin{equation}
\mu_{B_r(x^*)}(B_r(x^*) \cap G^{-n} B_r(x^*)) \leq |\log (r)|^{-A} \;\;\;\forall n\leq K|\log (r)|.
\end{equation} 
For $A>1$, summing over $n\in [1,K|\log (r)|]$ yields
\begin{equation*}
\mu_{B_r(x^*)}(\f{B_r(x^*),G}\leq K|\log(r)|)\rightarrow 0 \;\;\;\textit{ as }r\rightarrow 0.
\end{equation*}
Since this is true for all $K>0$, we can easily replace $K$ by some $K(r)\nearrow\infty$ growing slowly enough. This is a routine argument which is left to the reader.
} for $G$ at almost every $x^*$. Since $\tau$ is bounded, $T$ also has this property.

\item It is shown in \cite[Lemma 5]{BS} that, for a map of positive entropy, condition (LR) is satisfied at almost every point (In fact \eqref{ballretasm:lr} is satisfied for all $y\in B_r(y^*)$). This remains true for maps of the form $T(x,y)=G_{\tau(y)}(x)$, (in this case (LR') is satisfied) for bounded $\tau$, where $G$ has positive entropy.
\item Considering the previous remarks it may seem unnecessary to state condition (SLR). 
Note however that none of the conditions can be satisfied at periodic points, and the maps we want to use for $T$ will have plenty periodic points. (SLR) will be useful to show the PLT \textbf{everywhere}, if we can choose $R$ without periodic points.
\item In most of the examples (see \S \ref{ScEx}) we will have 
\begin{equation*}
\left|\left|\sum_{j=0}^{n-1} f\circ R^j -n\int f \d\nu\right|\right|_{L^2(\nu)} \leq C||f||_{H^{r'}} n^{\delta} \;\;\;\forall f\in H^{r'}, \; \forall n\geq 1.
\end{equation*}
Since $C^{r'}\subset H^{r'}$ and
 $||f||_{H^{r'}}\leq ||f||_{C^{r'}}$ for $f\in C^{r'}$, this implies condition $(EE)$.
\end{enumerate}
\end{remark}

\begin{theorem}\label{plt:1}
Assume that $S(x,y)=(T(x,y),R(y))$ satisfies conditions (MEM), (EE), and (BR$(x^*,y^*)$) for some $(x^*,y^*)\in X\times Y$. If
\begin{equation}\label{dimdec}
d>3d'\frac{r'+1}{1-\delta}
\end{equation}
then $S$ satisfies the PLT at $(x^*,y^*)$.
\end{theorem}

\begin{corollary}\label{plt1cor}
If $T(x,y)=T(x)$ preserves a smooth measure and satisfies (MEM), and $R$ satisfies (EE), then $S=T\times R$ satisfies the PLT almost everywhere.
\end{corollary}
If $T$ preserves a smooth measure, then, by \cite{emimpliesbernoulli}, $T$ is Bernoulli, in particular, it has positive entropy. (NSR$(x^*)$) and (LR'$(x^*)$) are satisfied almost everywhere by Remark \ref{conditionsrem}.\\
\\

For some applications it will be useful to choose $T(x,y)=G_{\tau(y)}(x)$, where $\int_Y \tau \d\nu=0$. However, in this case, $T$ will not satisfy condition (MEM). Fortunately we can apply similar techniques if ergodic averages of $\tau$ grow faster than logarithmically. More explicitly denote $\tau_n=\sum_{j=0}^{n-1} \tau\circ R^j$, assume there is a $\zeta:\mathbb{N}\rightarrow(0,\infty)$ with $\log(n)=o(\zeta(n))$ and a $\kappa>0$ such that
\begin{equation*}\tag{BA}
\nu(|\tau_n|<\zeta(n))\leq O(n^{-\kappa}).
\end{equation*}

\begin{theorem}\label{plt:2}
Assume that $S(x,y)=(T(x,y),R(y))$, where $T(x,y)=G_{\tau(y)}(x)$, satisfies conditions (MEM) with $G$ instead of $T$. Suppose that $R$ satisfies (EE), and $\tau$ satisfies (BA).
Let $x^*\in X$, $y^*\in Y$. If there is a $\delta_2>0$ such that for small enough $\rho>0$ we have
\begin{equation}\label{bigreturnsinplt:2}
\f{B_{\rho}(y^*)}\geq \rho^{-\delta_2} \;\;\;\textit{on }B_{\rho}(y^*),
\end{equation}
and
\begin{equation}\label{dimdec:2}
d>3d'\frac{r'+1}{1-\delta}\;\;\;\textit{ and }\;\;\; \kappa>\frac{d'}{\delta_2}.
\end{equation}
then $S$ satisfies the PLT at $(x^*,y^*)$.
\end{theorem}

\begin{remark}\label{pointwiseEEplt}
If in the situation of Theorem \ref{plt:1} resp. \ref{plt:2}, instead of (EE), we assume a stronger pointwise version\\
\\
(EE') There are $r'>0$ and $0<\delta<1$ such that 
\begin{equation}\label{ballretam:ee'}
\left|\sum_{j=0}^{n-1} f(R^j(y)) -n\int f \d\nu\right| \leq C||f||_{C^{r'}} n^{\delta} \;\;\;\forall f\in C^{r'}, y\in Y, n\geq 1.
\end{equation}
then we can weaken the assumption \eqref{dimdec} resp. \eqref{dimdec:2} on the dimension to
\begin{equation*}
d>d'\frac{r'+1}{1-\delta}.
\end{equation*}
To see this, note that the content of Proposition \ref{l2boundprop} is to deduce (a weaker version of) \eqref{ballretam:ee'}, with $\frac{2+\delta}{3}$ instead of $\delta$, from \eqref{ballretasm:ee}.

\end{remark}

\section{Examples}
\label{ScEx}
The definitions of the maps in Examples \ref{explethm:1}, \ref{explethm:2} and Lemma \ref{existstau}  are given in \S \ref{explesec}.
For most of the examples we present, the choice of $R$ is more interesting then the choice of $T$, mostly because (MEM) implies chaotic behavior and so the PLT in that setting is not surprising.
 We will thus not focus too much on $T$ for this section. We only present some examples here, there are many others one can verify using Theorem \ref{plt:1}.
\begin{example}\label{explethm:1}
Let $T$ be a map satisfying (MEM) on a manifold of sufficiently high\footnote{Sufficient bounds are given in \S \ref{explesec}.} dimension then 
\begin{enumerate}[(i)]
\item if $R$ is a diophantine rotation, then $T\times R$ satisfies the PLT everywhere;
\item if $R$ is the time $1$ map of a horocycle flow on $\Gamma\backslash SL_2(\mathbb{R})$ where
$\Gamma$ is a cocompact lattice, then $T\times R$ satisfies the PLT everywhere;
\item if $R$ is a skew-shift, then $T\times R$ satisfies the PLT everywhere.
\end{enumerate}
\end{example}

\begin{remark}
\begin{enumerate}[(i)]
\item In \S \ref{explesec} we will show that the map $R$ from example \ref{explethm:1}(i)-(iii) satisfies (EE) and (SLR($y^*$)) for every $y^*$. The conclusion then follows from Theorem \ref{plt:1}.
\item The PLT \textbf{almost everywhere} can be shown more readily. \\
By Corollary \ref{plt1cor}, we just have to check (EE) for the map $R$, which holds for a big class of maps, Examples will be given in \S \ref{EESec}.
\end{enumerate}
\end{remark}

Theorem \ref{plt:2} can be used to construct $T,T^{-1}$ transformations of zero entropy that satisfy the PLT. All that remains to do is to construct a $\tau$ satisfying (BA),
 this can be done with the construction given in \cite[Proposition 3.9]{DDKN}.

\begin{lemma}\label{existstau}
Let $R_{\alpha}:\mathbb{T}^{d'}\rightarrow\mathbb{T}^{d'}$ be a diophantine rotation, i.e 
\begin{equation*}\tag{D}
\left| \langle k, \alpha \rangle -l\right| >C|k|^{-\lambda} \;\;\;\forall k\in \mathbb{Z}^{d'},k\not=0,l\in\mathbb{Z},
\end{equation*}
for some $\lambda\geq 1$. For $\frac{n}{2}<\rho<d'$ there is a $d\geq 1$ and a function $\tau\in C^{\rho}(\mathbb{T}^{d'},\mathbb{R}^d)$ such that $\nu(\tau)=0$, while
\begin{equation*}
\nu(||\tau_n||<\log^2(n))=o(n^{-5}).
\end{equation*}
\end{lemma}

Note that in order to apply Theorem \ref{plt:2} we can always make $d$ as big as we want.

\begin{example}\label{explethm:2}
Let $R=R_{\alpha}$ be a diophantine rotation with $d'=2$ and $\lambda=2$, $\tau$ be the function from Lemma \ref{existstau}, and let $G$ be the Weyl Chamber flow on $SL(d,\mathbb{R}) / \Gamma$, where $\Gamma$ is a uniform lattice. If $d>10$ then $S(x,y)=(G_{\tau(y)}(x),R_{\alpha}(y))$ satisfies the PLT everywhere.
\end{example}

\section{The delayed PLT}

The main step in the proof will be to show a generalised version of the PLT (for fiberwise returns), along a subsequence, this is what we will call a delayed PLT.\\
This 'delayed PLT' in itself is of independent interest, so let us make a more general statement. Let $\alpha=(\alpha^{(n)})_{n\geq 1}$ be a sequence of natural numbers, we will refer to $\alpha$ as the \textit{delay sequence}, and denote $\tilde{\alpha}^{(n)}=\sum_{j=1}^n \alpha^{(n)}$. For measurable $A\subset X$ we define the \textit{delayed consecutive return times to $A$ along $\alpha$} as
\begin{equation}\label{retdef:delayed}
\begin{aligned}
&\f{A,\alpha}(x):=\f{A,\alpha}^{(1)}(x):=\min(j\geq 1\;|\; T^{\tilde{\alpha}^{(j)}}(x)\in A)\\
&\f{A,\alpha}^{(n+1)}(x):=\min(j\geq 1\;|\;T^{\tilde{\alpha}^{(\f{A,\alpha}^{(1)}(x)+\f{A,\alpha}^{(2)}(x)+\dots+\f{A,\alpha}^{(n)}(x)+j)}}(x)\in A)\\
&\Phi_{A,\alpha}:=(\f{A,\alpha}^{(1)},\f{A,\alpha}^{(2)},...).
\end{aligned}
\end{equation}

The main example in this paper will be $\alpha^{(n)}=\varphi_B^{(n)}(y)$ for some $y\in Y$ and $B\subset Y$\footnote{As in \S \ref{secThestrategy}, see especially \eqref{delaypoiscon:1}.}, however other choices are of interest, for example $\alpha^{(n)}=g(n)$, where $g$ is a polynomial, or $\alpha^{(n)}=p_n$, where $p_n$ is the $n$th prime. \\ 
Given a rare sequence $\s{A}{l}$, we will distinguish between two cases
\begin{itemize}
\item[1)] the delay sequence is fixed in $l$,
\item[2)] the delay sequence is allowed to vary with $l$.
\end{itemize}

\begin{definition}
\begin{enumerate}[(i)]
\item Let $x^*\in X$ and, for $r>0$, denote by $B_r(x^*)$ the geodesic ball of radius $r$ centred at $x^*$. 
Let $\alpha=(\alpha^{(n)})_{n\geq 1}$ be a sequence of natural numbers. We will say that \textit{$T$ satisfies the delayed PLT along $\alpha$ at $x^*$} if
\begin{equation*}
\mu(B_r(x^*))\Phi_{B_r(x^*),\alpha}\xRightarrow{\mu} \Phi_{Exp} \;\;\;\textit{ as }r\rightarrow 0.
\end{equation*}
\item Let
$PLT(\alpha)=\{x^*\in X \;|\; T$  satisfies the delayed PLT along $\alpha$ at $x^*\}.
$

If $x^*\in PLT(\alpha)$ for all sequences of natural numbers $\alpha$, then we say that \textit{$T$ satisfies the delayed PLT at $x^*$}.
\end{enumerate}
\end{definition}

The analogous definition for varying $\alpha$ is
\begin{definition}
\begin{enumerate}[(i)]
\item Let $x^*\in X$ and, for $r>0$, denote by $B_r(x^*)$ the geodesic ball of radius $r$ centred at $x^*$. 
Let $\alpha=((\alpha_r^{(n)})_{n\geq 1})_{r>0}$ be a collection of sequences of natural numbers. We will say that \textit{$T$ satisfies the varying delayed PLT along $\alpha$ at $x^*$} if
\begin{equation*}
\mu(B_r(x^*))\Phi_{B_r(x^*),\alpha_r}\xRightarrow{\mu} \Phi_{Exp} \;\;\;\textit{ as }r\rightarrow 0.
\end{equation*}
\item Let
$PLT(\alpha)=\{x^*\in X \;|\; T$  satisfies the varying delayed PLT along $\alpha$ at $x^*\}.
$
If $x^*\in PLT(\alpha)$ for all sequences of natural numbers $\alpha$, then we say that \textit{$T$ satisfies the delayed PLT at $x^*$}.
\end{enumerate}
\end{definition}

In case 1) the main result is a straightforward modification of Theorems \ref{plt:1} and \ref{plt:2}.
\begin{theorem}\label{dplt:1}
Let $\alpha$ be a sequence of positive integers. Assume the conditions of Theorem \ref{plt:1} or \ref{plt:2} hold replacing (EE) in both cases by
\begin{equation*}
\left|\left|\sum_{j=0}^{n-1} f\circ R^{\tilde{\alpha}^{(j)}} -n\int f \d\nu\right|\right|_{L^2(\nu)} \leq C||f||_{C^{r'}} n^{\delta} \;\;\;\forall f\in C^{r'}, \; \forall n\geq 1.
\end{equation*}
Then $S$ satisfies the delayed PLT along $\alpha$ at $(x^*,y^*)$.
\end{theorem}

The proof is analogous to the proof of Theorem \ref{plt:1} resp. \ref{plt:2} 
(with $\Phi_{B_l}$ replaced by $\Phi_{B_l,\alpha}$ and Lemma \ref{weakapproxlem} replaced by Lemma \ref{robustnesslem}).
Therefore no detailed proof will be given
\footnote{Note however that there is no relation of delayed hitting times to delayed return times as in Theorem \ref{hitretrelcor}.}.\\
\\

In case 2) the main result is a special case of Proposition \ref{delaykapretprop}(III), however it is worth stating by itself.
\begin{theorem}\label{dplt:2}
Suppose $T$ satisfies conditions (MEM), (SLR$(x^*)$) and (NSR$(x^*)$), then $T$ satisfies the varying delayed PLT at $x^*$.
\end{theorem}

\part{Proofs}

\section{Cumulative return times}

Most of the statements (and proofs) below are much more convenient to state in terms of cumulative return times. For a measurable set $A\subset X$ (or $B\subset Y, C\subset X\times Y$) the \textit{sequence of cumulative return times} to $A$ is given by
\begin{align*}
&\sigma_A^{(n)}=\sum_{j=0}^{n-1} \f{A} \circ T_A^j , \;\;\; n\geq 1\\
&\Sigma_A=(\sigma_A^{(1)},\sigma_A^{(2)},\dots),
\end{align*}
and similar notation for delayed returns.
\\
\\

When studying distributional convergence of $\Phi_A$, one can equivalently study for distributional convergence of $\Sigma_A$. \\
Indeed, let $\iota:[0,\infty)^{\mathbb{N}}\rightarrow [0,\infty)^{\mathbb{N}}$ the map
\begin{equation*}
\iota(x_1,x_2,x_3,...)=(x_1,x_1+x_2,x_1+x_2+x_3,\dots).
\end{equation*}
Since $\iota$ is a homeomorphism, standard theory shows that 
\begin{equation}\label{phisigmaeqpuiv}
\mu(A_l)\Phi_{A_l} \xRightarrow{\mu} \Phi \; \textit{ if and only if } \; \mu(A_l) \Sigma_{A_l} \xRightarrow{\mu} \iota(\Phi),
\end{equation}
where we use the obvious extension of $\iota$ to $[0,\infty]^{\mathbb{N}}$. Denote $\Sigma_{Exp}=\iota(\Phi_{Exp})$.

\section{PLT along varying subsequences}\label{PLTsubssec}
\subsection{Approximation}
In our work we often need to apply the mixing condition (MEM), and the quantitative ergodicity (EE), for indicator functions, hence we have to approximate them by functions in $C^r$ resp. $C^{r'}$. 
\begin{definition}
\label{DefApprox}
Let $M$ be a $C^r$ manifold with dimension $dim(M)=d$, $\lambda$ be a measure on $M$.

Let $B_{\alpha}\subset M$ be measurable subsets for $\alpha$ in some index set. We say that $\{B_{\alpha}\}$ is \textit{regularly approximable in $C^r$} if there is a constant $C>0$ such that, for each $\alpha$ and for $0<\epsilon<\frac{1}{10}\lambda(B_{\alpha})^{\frac{1}{d}}$, there are $\overline{h},\underline{h}\in C^r$ with $\underline{h}\leq 1_B\leq \overline{h}$ and
for $h\in \{\overline{h},\underline{h}\}$
\begin{equation}\label{regapproxeq}
\DS \lambda(1_{B_{\alpha}}\not= h)\leq\lambda(B_{\alpha})^{\frac{d-1}{d}}\epsilon, \;\textit{ while }\; ||h||_{C^r}\leq C\epsilon^{-r}.
\end{equation}
We call the least such constant $C>0$ the \textit{approximant of $B$} and denote it by $\app(\{B_{\alpha}\})$.

\end{definition}

\begin{lemma}\label{approxlem:2}
Let $M$, $r$, $d$ and $\lambda$ be as in Definition \ref{DefApprox}. 
Assume in addition  that $\lambda$ is absolutely continuous w.r.t volume with bounded density. Suppose that $\bigcup_{\alpha} B_{\alpha}$ is relatively compact, and there is an open $U\supset \overline{\bigcup_{\alpha} B_{\alpha}}\supset\bigcup_{\alpha} B_{\alpha}$ and a $C^r$-diffeomorphism $\iota:U\rightarrow V$ for some open set $V\subset \mathbb{R}^d$ such that each $\iota(B_{\alpha})$ is a ball. Then $\{B_{\alpha}\}$ is regularly approximable in $C^r$.
\end{lemma}

\begin{proof}
(i) We may assume that $M=\mathbb{R}^d$ and $\lambda$ is Lebesgue measure, otherwise we pick up another constant, which can be absorbed into $C$. Furthermore we can assume that all $B_{\alpha}$ are balls centered at the origin. In the following fix $\alpha$ and, dropping the $\alpha$ from our notation, let $B=B_t(0)$. We have $\lambda(B)=C_1 t^d$, where $C_1$ is the volume of the $d$-dimensional unit ball.
\\

\noindent(ii) Let $\theta:\mathbb{R}\rightarrow [0,1]$ be a smooth function with $\theta(x)=1$ if $x<0$, and $\theta(x)=0$ if $x>1$. For $t>\epsilon>0$ consider
\begin{equation*}
\hat{\theta}(x)=\theta(\epsilon^{-1}(x-t))
\end{equation*}
then $\hat{\theta}$ is still smooth and $||\hat{\theta}||_{C^r}=\epsilon^{-r}||\theta||_{C^r}$. Consider
$\overline{h}:\mathbb{R}^d\rightarrow[0,1]$ given by
$\DS
\overline{h}(x)=\hat{\theta}(|x|),$
then 
\begin{itemize}
\item $\overline{h}$ is smooth, away from the origin because it is the composition of smooth functions, and near the origin it is constant $1$,
\item $\overline{h}(x)=1$ if $|x|<t$ and $\overline{h}(x)=0$ if $|x|>t+\epsilon$,
\item and $||\overline{h}||_{C^r}\leq C_3 \epsilon^{-r}$
where $C_3=rC_2||\theta||_{C^r}$ and 
$C_2$ is the $C^r$ norm of the smooth function $x\mapsto |x|$ on $\{t\leq |x|\leq t+\epsilon\}$. 
\end{itemize}
Furthermore we have
\begin{align*}
\lambda(\overline{h}\not = 1_B)&=\lambda(t\leq |x|\leq t+\epsilon)=C_1((t+\epsilon)^d-t^d)\\
&\leq C_1 d t^{d-1}\epsilon\leq d C_1^{\frac{1}{d}} \lambda(B)^{\frac{d-1}{d}} \epsilon,
\end{align*}
all the constants can be absorbed in the constant $C$ from the claim, the constant only depends or $r,d$. \\
For $\underline{h}$ repeat the calculations with $\hat{\theta}=\theta(\epsilon^{-1}(x-(t-\epsilon)))$ instead.
\end{proof}

\subsection{Proof of the PLT along varying subsequences}
For our purposes it will not be enough to consider the delayed PLT for a single rare sequence $\s{A}{l}$, rather let $K\geq 1$, and $A_l^{(1)},...,A_l^{(K)}$ be subsets of $X$ such that $\{A_l^{(k)}\}$ is regularly approximable in $C^r$. Assume that there are $\omega^{(1)},...,\omega^{(K)}>0$ and $r_l\rightarrow 0$ such that
\begin{equation}
\label{OmegaK}
\mu(A_l^{(k)})=\omega^{(k)} r_l^d + o(r_l^d).
\end{equation} 

Given $\kappa_l^{(j)}\in \{1,...,K\}$, for $l,j\geq 1$, define the \textit{cumulative return times} by
\begin{align*}
&\sigma_{\kappa_l,\alpha_l,y}^{(1)}(x)=\min(j\geq 1\;|\; T_y^{\tilde{\alpha}_l^{(j)}}(x)\in A_l^{(\kappa_l^{(j)})}),\\
&\sigma_{\kappa_l,\alpha_l,y}^{(n+1)}(x)=\min(j\geq \tau_{\kappa_l,\alpha_l,y}^{(n)}(x)+1\;|\; T_y^{\tilde{\alpha}_l^{(j)}}(x)\in A_l^{(\kappa_l^{(j)})}),\\
&\Sigma_{\kappa_l,\alpha_l,y}=(\sigma_{\kappa_l,\alpha_l,y}^{(1)},\sigma_{\kappa_l,\alpha_l,y}^{(2)},...).
\end{align*}

Denote the frequency with which $A_l^{(\kappa_l)}=A_l^{(k)}$ by
\begin{equation*}
p_{l,t}^{(k)}:=\frac{1}{t}\#\{j=1,...,t_l\;|\;\kappa_l^{(j)} =k\}.
\end{equation*}

For now suppose that there are positive constants $\theta^{(k)}>0$ such that, for all $k=1,...,K$, and for all $t_l\nearrow \infty$ with $t_l=O(r_l^{-d})$,

\begin{equation}\label{kapprobcon}
p_{l,t_l}^{(k)} \rightarrow \frac{\theta^{(k)}}{\sum_{j=1}^K \theta^{(j)}}=:p^{(k)}\;\;\; \as{l}.
\end{equation}
Later on, when we will prove PLT in product systems, we will choose $\kappa$ and $\theta$ in a specific way\footnote{Say we want to prove a PLT for the system $T\times R$ and sets of the form $B_r(x^*)\times B_r(y^*)$ then we choose $K=1$ and $\theta>0$ such that $\nu(B_r(y^*))=\theta r^{d'} + o(r^{d'})$.} and \eqref{kapprobcon} will be satisfied by Lemma \ref{kaplem}.
\\
The main estimate of mixing rates for regularly approximable sets is the following.

\begin{lemma}\label{decorrlem}
Let $m\geq 1$, $A^{(1)},...,A^{(k)}\subset X$ be regularly approximable in $C^r$,  and $1\leq n_1<...<n_k$, then
\begin{equation*}
\left|\mu\left(\bigcap_{i=1}^k T^{-n_i}A^{(i)}\right)-\prod_{i=1}^k \mu (A^{(i)})\right|\leq K 
\max_{i=1,...,k} \mu(A^{(i)})^{\frac{d-1}{d} \frac{kr}{kr+1}} 
e^{-\frac{\gamma p}{kr+1} },
\end{equation*}
where $\displaystyle p=\min_{i=1,...,k-1} |n_{i+1}-n_i|$ and the constant $K>0$ only depends on $k$ and $\app(\{A^{(1)},...,A^{(k)}\})$.
\end{lemma}
\begin{proof}
Let $C=\app(\{A^{(1)},...,A^{(k)}\})$. By Lemma \ref{approxlem:2}, for every $\epsilon>0$, there are $h^{(i)}\in C^r$ such that $0\leq h^{(i)}\leq 1_{A^{(i)}}$ and
\begin{equation*}
\mu(1_{A^{(i)}}\neq h^{(i)})\leq \mu({A^{(i)}})^{\frac{d-1}{d}}\epsilon, \; \textit{ while } ||h^{(i)}||_{C^r}<C\epsilon^{-r}.
\end{equation*}
We estimate 
\begin{align*}
\left|\mu\left( \bigcap_{i=1}^k T^{-n_i}{A^{(i)}}\right)-\prod_{i=1}^k \mu (A^{(i)})\right|&\leq \left|\mu\left( \bigcap_{i=1}^k T^{-n_i}{A^{(i)}}\right)-\int_X \prod_{i=1}^k h^{(i)}\circ T^{n_i}\d\mu \right|\\
&+\left| \int_X \prod_{i=1}^k h^{(i)}\circ T^{n_i}\d\mu - \prod_{i=1}^k \int_X  h^{(i)}\circ T^{n_i}\d\mu \right|\\
&+ \left|\prod_{i=1}^k \int_X h^{(i)}\circ T^{n_i}\d\mu -\prod_{i=1}^k \mu (A^{(i)})\right|\\
&\leq 4k \max_{i=1,...,k} \mu(A^{(i)})^{\frac{d-1}{d}}\epsilon +C^k e^{-\gamma p} \epsilon^{-kr}.
\end{align*}
This bound is optimised for
\begin{equation*}
\epsilon^*=\left( \frac{C^k r}{4k} \max_{i=1,...,k} \mu(A^{(i)})^{-\frac{d-1}{d}} e^{-\gamma p}\right)^{\frac{1}{kr+1}},
\end{equation*}
so
$$
\left|\mu\left( \bigcap_{i=1}^k T^{-n_i}A^{(i)}\right)-\prod_{i=1}^k \mu (A^{(i)})\right|\leq \hat{C}\max_{i=1,...,k} \mu(A^{(i)})^{\frac{d-1}{d} \left(1-\frac{1}{kr+1}\right)} e^{-p\frac{\gamma}{kr+1}}
$$$$
+ \bar{C}\max_{i=1,...,k} \mu(A^{(i)})^{\frac{kr}{kr+1}\frac{d-1}{d}} e^{-\gamma p\left(1-\frac{kr}{kr+1}\right)}$$
$$\leq K \max_{i=1,...,k} \mu(A^{(i)})^{\frac{d-1}{d} \frac{kr}{kr+1}} e^{-\frac{\gamma}{kr+1} p},
$$
where the constants $\hat{C},\bar{C},K>0$ only depend on $k$ and on $C$.
\end{proof}

\begin{proposition}\label{delaykapretprop}
Suppose that $T$ satisfies (MEM), $(A_l^{(1)}),...,(A_l^{(K)})$ are sequences of rare events with $\mu(A_l^{(k)})=\omega^{(k)} r_l^d + o(r_l^d)$, for $K\geq 1$,
$\kappa_l$ satisfy \eqref{kapprobcon}  
for some $p^{(k)}>0$ with $\DS \sum_{k=1}^K p^{(k)}=1$, and let $\alpha_l=(\alpha_l^{(n)})_{n\geq 1}$ be sequences of natural numbers. Denote $A_l=\bigcup_{k=1}^K A_l^{(k)}$ and suppose that either
\begin{enumerate}[(I)]
\item $\alpha_l$ grows faster than $|\log\mu(A_l)|$ in the sense that
\begin{equation}
\label{GapLogLarge}
|\log\mu(A_l)|=o(\min_{n\geq 2}|\alpha^{(n)}_{l}|),
\end{equation} 
\item short returns to $A_l$ are rare in the sense that
\begin{equation}\label{noshortreturns:2}
\mu_{A_l}(\f{A_l}\leq a_l)\rightarrow 0 \;\;\;\as{l},
\end{equation}
for some sequence $\s{a}{l}$ with $|\log(\mu(A_l))|=o(a_l)$, and $\alpha_l$ grows at least as fast $|\log\mu(A_l)|$ in the sense that
\begin{equation}
\label{GapLogLarge:2}
|\log\mu(A_l)|=O(\min_{n\geq 2}|\alpha^{(n)}_{l}|),
\end{equation} 
\item or short returns to $A_l$ are rare in the sense that
\begin{equation}\label{noshortreturns:3}
\mu_{A_l}(\f{A_l}\leq a_l)\rightarrow 0 \;\;\;\as{l},
\end{equation}
for some sequence $\s{a}{l}$ with $|\log(\mu(A_l))|=o(a_l)$, and 
returns are at least logarithmically large, i. e there exists $c>0$ such that
\begin{equation}\label{GapLogLarge:3}
\f{A_l}(x) \geq c \log(\mu(A_l)) \;\;\;\mu-\textit{a.e } x\in A_l.
\end{equation}
\end{enumerate}
Then for $\nu$-a.e $y\in Y$
\begin{equation*}
\Omega r_l^d \Sigma_{\kappa_l,\alpha_l,y}\xRightarrow{\mu} \Sigma_{Exp} \;\;\;\as{l},
\end{equation*}
where $\Omega=\sum_{k=1}^K \omega^{(k)} p^{(k)}$.
\end{proposition}
\begin{proof}
Fix $y$ as in (MEM) and denote $T^n=T_y^n$.\\
(i) Taking a subsequence if necessary, we may assume that there is a $[0,\infty]$-valued process $\Sigma=(\sigma^{(1)},\sigma^{(2)},...)$ such that
\begin{equation*}
\Omega r_l^d \Sigma_{\kappa_l,\alpha_l}\xRightarrow{\mu} \Sigma \;\;\;\as{l}.
\end{equation*}
For $J\geq 1$ and $0<t_1<...<t_J$ we show 
\begin{equation}\label{waitingcond}
\mathbb{P}\left(\sigma^{(j)} \leq t_j, \; j=1, \dots, J\right)=\mathbb{P}\left(\sigma^{(j)}_{Exp} \leq t_j, \; j=1,\dots ,J\right).
\end{equation}
The trick is is to look at the "dual" object
\begin{equation*}
S_{\kappa_l,\alpha_l}^{(n)}=\sum_{j=1}^n 1_{A_l^{\kappa_l^{(j)}}}\circ T^{\tilde{\alpha}_l^{(j)}}.
\end{equation*}
The important relation here is the following
\begin{equation}
\label{Renewal}
S_{\kappa,\alpha_l}^{(n)}\geq N \; \iff\; \sigma_{\kappa,\alpha_l}^{(N)}\leq n.
\end{equation}
Indeed, evaluating both sides at some $x\in X$, the left side says that there are at least $N$ different times $1\leq j_1<...<j_N\leq n$ such that $T^{\tilde{\alpha}_l^{(j)}}(x)\in A_l^{\kappa_l^{(j)}}$. The right hand side expresses that, if $1\leq j_1<...<j_N$ are the first $N$ times that $T^{\tilde{\alpha}_l^{(j)}}(x)\in A_l^{\kappa_l^{(j)}}$, then $j_1+ \sum_{i=2}^N (j_{i}-j_{i-1}) \leq n$.
Let $(P_t)_{t\geq 0}$ be a Poisson process, such that $\Sigma_{Exp}$ are the cumulative waiting times of $(P_t)$, i.e $(P_t)$ and $\Sigma_{Exp}$ are related by 
\begin{equation*}
P_t\geq N\iff \sigma_{Exp}^{(N)} \leq t.
\end{equation*} 
The right side of \eqref{waitingcond} is equal to
\begin{equation*}
\mathbb{P}\left(\sigma_{Exp}^{(k)} \leq t_k, \; k=1,...,J\right)=\mathbb{P}(P_{t_k}\geq k, \; k=1,...,J).
\end{equation*}
Due to \eqref{Renewal} it is enough to show
\begin{equation*}
\left(S_{\kappa_l,\alpha_l}^{\ffloor{t_1}{\Omega r_l^d}},S_{\kappa_l,\alpha_l}^{\ffloor{t_2}{\Omega r_l^d}},...,S_{\kappa_l,\alpha_l}^{\ffloor{t_J}{\Omega r_l^d}}\right)\xRightarrow{\mu} (P_{t_1},...,P_{t_{J}}) \;\;\;\as{l}.
\end{equation*}

(ii) 
Taking a further subsequence if necessary, there are $[0,\infty]$-valued $\tilde{P}_{t_1},...,\tilde{P}_{t_J}$ such that 
\begin{equation*}
\left(S_{\kappa_l,\alpha_l}^{\ffloor{t_1}{\Omega r_l^d}},S_{\kappa_l,\alpha_l}^{\ffloor{t_2}{\Omega r_l^d}},...,S_{\kappa_l,\alpha_l}^{\ffloor{t_J}{\Omega r_l^d}}\right)\xRightarrow{\mu} (\tilde{P}_{t_1},...,\tilde{P}_{t_J}) \;\;\;\as{l}.
\end{equation*}
We will show that
\begin{enumerate}[(A)]
\item $\tilde{P}_{t_k}-\tilde{P}_{t_{k-1}}$ is Poisson distributed with intensity $t_k-t_{k-1}$ for $k=1,...,d$,
\item and 
$
(\tilde{P}_{t_1}-\tilde{P}_{t_0},\tilde{P}_{t_2}-\tilde{P}_{t_1},...,\tilde{P}_{t_J}-\tilde{P}_{t_{J-1}})
$
is an independent vector,
where $t_0=0$.\footnote{This is essentially Watanabe's characterisation of Poisson-processes.} 
\end{enumerate}
Clearly $P_0=\tilde{P}_0=0$.\\
For $j=1,\dots,J$
denote $S_{t_j,l}=S_{A_l,\kappa_l}^{\ffloor{t_j}{\Omega r_l^d}}$. Assertions (A) and (B) will follow\footnote{Here we apply the method of moments, see eg \cite[Theorem 30.2]{Billprobaandmeasure}.} 
once we show that, for all $m_1,...,m_d\geq 1$,
\begin{equation*}
\int_X \prod_{j=1}^{J} {{S_{t_j,l}-S_{t_{j-1},l}} \choose m_j} \d\mu = \prod_{j=1}^{J} \frac{(t_j-t_{j-1})^{m_j}}{m_j!} +o(1) \;\;\;\as{l}.
\end{equation*}
In the rest of the proof fix $J\geq 1$, $0=t_0<t_1<\dots<t_J$ and $m_1,\dots,m_J\geq 1$.\\

(iii)  First, for each $j=1,\dots,J$, rewrite 
\begin{equation*}
S_{t_j,l}-S_{t_{j-1},l}=\sum_{i=\ffloor{t_{j-1}}{\Omega r_l^d}+1}^{\ffloor{t_j}{\Omega r_l^d}} 
1_{A_l^{(\kappa_l^{(i)})}}\circ T^{\tilde{\alpha}_l^{(i)}}.
\end{equation*} 
So
\begin{equation}\label{slbinomprod}
\prod_{j=1}^{J} {{S_{t_j,l}-S_{t_{j-1},l}} \choose m_j}=\prod_{j=1}^d \sum_{\ffloor{t_{j-1}}{\Omega r_l^d}+1\leq k_{1,j} <...< k_{m_j,j} \leq \ffloor{t_j}{\Omega r_l^d}} \prod_{i=1}^{m_j} \xi_{i,j}
\end{equation}
where $\xi_{i,j}=1_{A_l^{(\kappa_l^{(k_{i,j})})}}\circ T^{\tilde{\alpha}_l^{({k_{i,j})}}}$.
To simplify notation we will also denote $m=(m_1+\dots+m_J)$, $\DS \omega=\min_{k=1,...,K} \omega^{(k)}$, $p_l=2m\frac{mr+1}{\gamma}\left|\log( \omega r_l^d)\right|$,
$$ \Delta_l:=\{\textbf{k}=(k_{i,j})_{\substack{j=1,...,J\\i=1,...,m_j}} \;\Big|\; \ffloor{t_{j-1}}{\Omega r_l^d}+1\leq k_{1,j}\!\!<\!\!\dots\!\!<\!\! k_{m_j,j} \leq\!\! \ffloor{t_j}{\Omega r_l^d} \;\textit{ for } j=1,...,J\}$$  and
$$ \Delta'_l:=\{\textbf{k}\in \Delta_l \;|\;\min_{\substack{j=1,...,J, \; i=1,..., m_j\\ j'=1,...,J, \; i'=1,..., m_{j'},\; (j,i)\not=(j',i')}} |\tilde{\alpha}_l^{(k_{i,j})}-\tilde{\alpha}_l^{(k_{i',j'})}|\leq p_l\} $$
We will split the sum in \eqref{slbinomprod} into two terms 
\begin{equation*}
\prod_{j=1}^J {{S_{t_j,l}-S_{t_{j-1},l}} \choose m_j} = M_l+R_l,
\end{equation*}
where 
\begin{equation*}
M_l =M_{(t_j),l,(m_j)}= \sum_{\textbf{k}\in \Delta_l\setminus \Delta'_l} \prod_{\substack{j=1,...,J \\ i=1,..., m_j}} 
\xi_{i,j}
,\quad
R_l =R_{(t_j),l,(m_j)} =\sum_{\textbf{k}\in \Delta'_l} \prod_{\substack{j=1,...,J \\ i=1,..., m_j}} 
\xi_{i,j}.
\end{equation*}
We will show that
\begin{equation}\label{mlrlcon}
\int_X M_l \d\mu \rightarrow \prod_{j=1}^{J} \frac{(t_j-t_{j-1})^{m_j}}{m_j!} \;\textit{ and }\; \int_X R_l\d\mu \rightarrow 0 \;\;\;\as{l}.
\end{equation}

(iv)
Let us first treat $M_l$. For $l\geq 1$ and $\textbf{k}\in \Delta_l$, by Lemma \ref{decorrlem} we have
$$
\left|\mu\left(\bigcap_{\substack{j=1,...,J \\ i=1,..., m_j}} T^{-\tilde{\alpha}_l^{(k_{i,j})}}
A_l^{(\kappa_l^{(k_{i,j})})}\right)-\prod_{j,i} \mu \left(A_l^{(\kappa_l^{(k_{i,j})})}\right)\right|
$$$$
\leq K \max_{j,i} \mu\left(A_l^{(\kappa_l^{(k_{i,j})})}\right)^{\frac{d-1}{d} \frac{mr}{mr+1}} 
e^{-\frac{\gamma \min_{i,j} \alpha_l^{(k_{i,j})}}{mr+1} }.
$$
For $\textbf{k}\in \Delta_l\setminus \Delta'_l$, this yields
$$
\mu\left(\bigcap_{\substack{j=1,...,J \\ i=1,..., m_j}} T^{-\tilde{\alpha}_l^{(k_{i,j})}}
A_l^{(\kappa_l^{(k_{i,j})})}\right)
=r_l^{md}\prod_{\substack{j=1,...,J \\ i=1,..., m_j}} \omega_l^{(\kappa_l^{k_{i,j}})} +o(r_l^{md}),
$$
and the $o$-term does not depend on $\textbf{k}$. Summing over $\textbf{k}\in \Delta_l\setminus \Delta'_l$, and using \eqref{kapprobcon}, yields
\begin{align*}
\int_X M_l \d\mu&=\int_X \sum_{\textbf{k}\in \Delta_l\setminus \Delta'_l} \prod_{\substack{j=1,...,J \\ i=1,..., m_j}} 1_{A_l^{(\kappa_l^{(k_{i,j})})}}\circ T^{\tilde{\alpha}_l^{({k_{i,j})}}} \d\mu\\
&= r_l^{md} \sum_{\textbf{k}\in \Delta_l\setminus \Delta'_l} \prod_{\substack{j=1,...,J \\ i=1,..., m_j}} \omega^{(\kappa_l^{(k_{i,j})})} +o(1)\\
&\overset{*}{=} r_l^{md} \sum_{\textbf{k}\in \Delta_l} \prod_{\substack{j=1,...,J \\ i=1,..., m_j}} \omega^{(\kappa_l^{(k_{i,j})})} +o(1)\\
&=r_l^{md} \prod_{j=1}^J \frac{1}{m_j !} \left(\sum_{k=1}^{K} \ffloor{t_j-t_{j-1}}{\Omega r_l^d} p^{(k)}_{l, \ffloor{t_j-t_{j-1}}{\Omega r_l^d}} \omega^{(k)} \right)^{m_j} +o(1)\\
&=\Omega^{-m} \prod_{j=1}^J \frac{(t_j-t_{j-1})^{m_j}}{m_j !} \left(\sum_{k=1}^K p^{(k)} \omega^{(k)}\right)^{m_j} +o(1)\\
&=\prod_{j=1}^{J} \frac{(t_j-t_{j-1})^{m_j}}{m_j!} +o(1),
\end{align*}
for $*$ note that $\# \Delta_l' = O(r_l^{-md+1} |\log(r_l)|)$.
This shows the first assertion of \eqref{mlrlcon}.

(v) 
In order to treat $R_l$, first note that under assumption $(I)$ we have $R_l=0$ for big enough $l$. 
In the following we focus on assumptions $(II)$ and $(III)$. Note that 
\begin{equation*}
\int_X R_l \d\mu=\int_X 1_{R_l\neq 0} R_l \d\mu\leq \mu(\supp(R_l)) \|R_l\|_{L^2} 
\end{equation*} 
and
\begin{equation*}
\supp(R_l) \subset \bigcup_{j=1}^{\ffloor{t_J}{\Omega r_l^d}} T^{-\tilde{\alpha}_l^{(j)}} \left(A_l  \cap \{\f{A_l}\leq 3p_l\}\right) =:U_l,
\end{equation*}
since $p_l=O(|\log(r_l)|)$, from \eqref{noshortreturns:2} resp. \eqref{noshortreturns:3}, it follows that
\begin{equation*}
\mu(\supp(R_l))\leq \mu(U_l) = O(r_l^{-d})\mu(A_l)o(1)=o(1).
\end{equation*}
Therefore,
in order to show $\int_X R_l \d\mu\rightarrow 0$, it is enough to show that $(R_l)_{l\geq 1}$ is bounded in $L^2$. Notice that
\begin{equation*}
R_l^2\leq \sum_{1\leq k_1,...,k_{2m} \leq \ffloor{t_J}{\Omega r_l^d}} \prod_{i=1}^{m} 1_{A_l^{(\kappa_l^{(k_i)})}} \circ T^{\tilde{\alpha}_l^{(i)}}=S_{t_J,l}^{2m}.
\end{equation*}
We may write 
\begin{equation*}
S_{t_J,l}^{2m}= \sum_{k=1}^{2m} \stirling{2m}{k} {S_{t_J,l} \choose k} \leq C_m \sum_{k=1}^{2m} {S_{t_J,l} \choose k} \leq C_m \sum_{k=1}^{2m} (M_{t_J,l,k}+R_{t_J,l,k}),
\end{equation*}
where $\stirling{2m}{k}$ are the Stirling numbers of the second kind and 
 $\DS C_m=\!\!\max_{k=1,...,2m} \stirling{2m}{k}$. Now the previous parts of the proof show that
\begin{equation*}
\int_X M_{t_J,l,m} \d\mu \;\;\; \textit{is bounded as }l\rightarrow \infty, \forall m\geq 1,
\end{equation*}
it remains to show that
\begin{equation*}
\int_X R_{t_J,l,m} \d\mu \;\;\; \textit{is bounded as }l\rightarrow \infty, \forall m\geq 1.
\end{equation*}

(vi)
In order to 
bound $\int_X R_{t_J,l,m} \d\mu $ for fixed $m\geq 1$
we first split up\footnote{Here the sets of sequences should be modified, i.e
$$ \Delta_l:=\{\textbf{k}=(k_{i})_{\substack{i=1,...,m}} \;\Big|\; 1\leq k_{1}\!\!<\!\!\dots\!\!<\!\! k_{m} \leq\!\! \ffloor{t_J}{\Omega r_l^d} \}$$  and
$$ \Delta'_l:=\{\textbf{k}\in \Delta_l \;|\;\min_{\substack{i=1,..., m\\ i'=1,..., m,\; i\not=i'}} |\tilde{\alpha}_l^{(k_{i})}-\tilde{\alpha}_l^{(k_{i'})}|\leq p_l\} $$
} $\Delta'_l$ into 
\begin{equation*}
\Delta_l^{(j)}:=\left\{\textbf{k}=(k_1,...,k_{m})\in \Delta_l\;\left|
\begin{aligned}
&\exists 1\leq i_i < ...< i_{m-j}\leq m \textit{ such that }\\
& |\tilde{\alpha}_l^{(k_{i+1})} - \tilde{\alpha}_l^{(k_{i})}|\leq \tilde{p}_l \; \forall i\in\{1,...,m\}\setminus\{i_1,...,i_{m-j}\},\\
&\textit{and } |\tilde{\alpha}_l^{(k_{i_r+1})}-\tilde{\alpha}_l^{(k_{i_r})}|> \tilde{p}_l \; \forall r=1,...,m-j
\end{aligned} 
\right. \right\},
\end{equation*}
for $j=1,...,m-1$, so that $\Delta'_l=\bigcup_{j=1}^{m-1} \Delta_l^{(j)}$. Under assumption $(II)$, 
because of \eqref{GapLogLarge:2}, there is a constant $c>0$ such that for $\rho=2m\frac{mr+1}{c \gamma}$ and big enough $l$ we have
\begin{equation*}
\# \Delta_l^{(j)}\leq {{m}\choose j} {\ffloor{t}{\Omega r_l^d} \choose {{m}-j}} \rho^j\leq m^{m} t^{m}\rho^{m} (\Omega^{-{m}}+1) r_l^{-d({m}-j)}.
\end{equation*}
On the other hand, for $\textbf{k}\in \Delta_l^{(j)}$, we can again use Lemma \ref{decorrlem} to estimate
\begin{equation*}
\int_X \prod_{s=1}^{m}  \xi_{s}\d\mu\leq \int_X \prod_{s\in \{1,...,{m}\}\setminus\{i_1,...,i_j\}} \xi_{s} \leq \Omega r_l^{d({m}-j)} +o(r_l^{d({m}-j)}),
\end{equation*}
where the $o$-term does not depend on $\textbf{k}$, and
\begin{equation*}
\xi_s=1_{A_l^{\kappa_l^{(k_s)}}}\circ T^{\tilde{\alpha}_l^{(k_s)}}.
\end{equation*}
Summing over $\textbf{k}\in \Delta_l'$ we obtain
\begin{equation*}
\int_X R_{t_j,l,m} \d\mu \leq 2m^{{m}+1} t^{m}\rho^{m} (\Omega^{-{m}+1}+\Omega), \;\;\;\forall m\geq 1,
\end{equation*}
and we conclude $\int_X R_l \d\mu \rightarrow0$ under assumption $(II)$.

(vii) Finally assume $(III)$. For fixed $j=1,...,m-1$, $l\geq 1$ and $1\leq i_1<...<i_{j}\leq \ffloor{t}{\Omega r_l^d}$ with $\DS \min_{s=1,...,j-1} |\tilde{\alpha}_l^{(i_{s+1})}-\tilde{\alpha}_l^{(i_s)}|\geq p_l$ set 
\begin{equation*}
A_{l,i_1,...,i_j}=\bigcap_{s=1}^j T^{-\tilde{\alpha}_l^{(i_s)}} A_l^{\kappa_l^{(i_s)}}.
\end{equation*}
By Lemma \ref{decorrlem}
\begin{equation*}
\mu(A_{l,i_1,...,i_j})\leq (\omega r_l)^{dj} +o(r_l^{dj}),
\end{equation*}
and the $o$ term doesn't depend on $(i_1,...,i_j)$. 
For $x\in A_{l,i_1,...,i_j}$ consider
\begin{equation*}
\mathfrak{K}_{l,i_1,...,i_j}(x)=\left\{\textbf{k}=(k_1,...,k_m)\in \Delta_l^{(j)}\;\left| \begin{aligned}
&\exists r_1,...,r_j\; \textit{ such that } k_{r_s}=i_s \; \forall s=1,...,j, \\
&\textit{ and }\; \prod_{s=1}^{m}  \xi_{s}(x)=1
\end{aligned} \right\} \right..
\end{equation*}
Then, since $\f{A_l}(x) \geq c \log(\mu(A_l))$, we have
\begin{equation*}
\# \mathfrak{K}_{l,i_1,...,i_j}(x)\leq \left(2\frac{p_l}{dc\log(r_l)}\right)^{m-j}=:\rho^{m-j},
\end{equation*}
since $p_l=constant * |\log(r_l)|$ this quantity doesn't depend on $l$.
At the same time we have 
\begin{equation*}
\supp\left( \prod_{s=1}^{m}  \xi_{s}\right) \subset A_{l,i_1,...,i_j} \;\;\;\textit{ if } \textbf{k}\in \mathfrak{K}_{l,i_1,...,i_j}(x) \; \textit{ for some } x.
\end{equation*}
Also every $\textbf{k}\in \tilde{\Delta}_l^{(j)}$ is in some $\mathfrak{K}_{l,i_1,...,i_j}(x)$, therefore
\begin{align*}
\sum_{\textbf{k}\in \tilde{\Delta}_l^{(j)}} \int_X \prod_{s=1}^{m}  \xi_{s} \d\mu &\leq \sum_{\substack{1\leq i_1<...<i_{j}\leq \ffloor{t}{\Omega r_l^d}\\ \min_{s=1,...,j-1} |\tilde{\alpha}_l^{(i_{s+1})}-\tilde{\alpha}_l^{(i_s)}|\geq p_l}} \int_{A_{l,i_1,...,i_j}} \# \mathfrak{K}_{l,i_1,...,i_j}(x) \d\mu(x)\\
&\leq \sum_{\substack{1\leq i_1<...<i_{j}\leq \ffloor{t}{\Omega r_l^d}\\ \min_{s=1,...,j-1}
|\tilde{\alpha}_l^{(i_{s+1})}-\tilde{\alpha}_l^{(i_s)}|\geq \tilde{p}_l}} 
\int_{A_{l,i_1,...,i_j}} \rho^j \d\mu \\
&\leq \rho^j \left(\omega r_l^{dj}+o(r_l^{dj}) \right) \left(\ffloor{t}{\Omega r_l^d}\right)^j\\ 
&\leq \rho^j t^j \frac{\omega}{\Omega^j} +o(1).
\end{align*}
Summing up over $j$ we get 
\begin{equation*}
\int_X R_l \d\mu \leq m \max(1,\rho^{m}) \max(1,t^{m}) \frac{\omega}{\max(1,\Omega^{m})} +o(1).
\end{equation*}
Following the argument in step (v) this shows $\int_X R_l \d\mu \rightarrow0$, hence \eqref{mlrlcon}, in case of assumption (III). This concludes the proof.
\end{proof}

\begin{remark}\label{delaykapretrem}
(i) Note that assumptions (I), (II), or (III) directly correspond to the three possible cases of 
(BR$(x^*,y^*)$)--(SLR$(y^*)$), (NSR$(x^*)$) AND (LR$(y^*)$), or (NSR$(x^*)$) AND (LR'$(x^*)$) respectively.

(ii)
In all of the examples we give in section \ref{ScEx}, the $\alpha_l$ will satisfy a condition stronger than (I). 
In fact, in this set-up, there is a $\delta_2>0$ such that
\begin{equation*}
\min_{i\geq 2}|\alpha^{(i)}_{l}|\geq \mu(A_l)^{-\delta_2},
\end{equation*}
compare also \eqref{bigreturnsinplt:2}\footnote{The constant $\delta_2$ given there is not exactly the same. In the notation there we have to use $\delta_2 \frac{d'}{d}$.}. If this stronger condition is satisfied instead of (I), then we do not need the full strength of exponential mixing in (MEM). Any superpolynomial rate will be enough. Details are given in the proof of 
Theorem~\ref{plt:2}, but we shall give a heuristic here.\\
When using mixing of all order with indicators of the form $1_{B_{r_l}(x^*)}$, the error term will contain a term coming from the $C^r$ norm in the definition of regularly approximable. 
In this case, using \eqref{regapproxeq}, this term will be of order $C r_l^{-drm}$. To compensate, say the rate of mixing is $\psi$, since the gaps $\alpha_l$ are large we can multiply with $\psi(\min_{i\geq 2}|\alpha^{(i)}_{l}|)$. So we want to show
\begin{equation*}
r_l^{-drm} \psi(r_l^{-d\delta_2})=o(1),
\end{equation*}
for all $m\geq 1$, thus $\psi$ should decay superpolynomially.

\end{remark}

\section{The PLT for rectangles.}
\subsection{The strategy}\label{secThestrategy}

Let us first outline the strategy of proving Theorem \ref{plt:1}. \\
Say we want to show the PLT for the geodesic balls $(Q_l)$ converging to $(x^*,y^*)$, then, for every $\epsilon>0$, we find suitable rectangles\footnote{Using the exponential map, this is a simple exercise in $\mathbb{R}^{d+d'}$. Here we use continuity of the densities of $\mu$ resp. $\nu$!} approximating $Q_l$ in the sense that $\bigcup_{k=1}^K A_l^{(k)}\times B_l^{(k)}\subset Q_l$ and
\begin{equation*}
(\mu\times\nu)_{Q_l}\left( Q_l \setminus \bigcup_{k=1}^K A_l^{(k)}\times B_l^{(k)} \right) <\epsilon.
\end{equation*}
Denote $Q_l':=\bigcup_{k=1}^K A_l^{(k)}\times B_l^{(k)}$.
We will use Proposition \ref{delaykapretprop} to show the PLT for $Q'_l$, i.e
\begin{equation*}
(\mu\times\nu)(Q'_l)\Sigma_{Q'_l}\xRightarrow{(\mu\times \nu)}\Sigma_{Exp}\;\;\;\as{l}.
\end{equation*}
Letting $\epsilon\rightarrow0$, the next Lemma will help us conclude that
\begin{equation*}
(\mu\times\nu)(Q_l)\Sigma_{Q_l}\xRightarrow{(\mu\times \nu)}\Sigma_{Exp}\;\;\;\as{l}.
\end{equation*}
\\
The critical approximation here is the following Lemma - see \cite[Theorem 4.4]{Z2} for a qualitative version, and the following statement can be easily deduced from the proof - for convenience we give a proof of this quantitative version in \S \ref{robustnesssec}.

\begin{lemma}\label{weakapproxlem}
There is a metric $D$ defined on the space of probability measures on $[0,\infty]^{\mathbb{N}}$ and modelling weak convergence of measures\footnote{
In the sense that $\lambda_l\Rightarrow\lambda$ if and only if $D(\lambda_l,\lambda)\rightarrow0$.} 
such that
\begin{equation*}
D\left(law_{(\mu\times\nu)_Q}((\mu\times\nu)(Q)\Phi_Q),law_{(\mu\times\nu)_{Q'}}((\mu\times\nu)(Q')\Phi_{Q'}\right)\leq 4(\mu\times \nu)_Q(Q\setminus Q'),
\end{equation*}
for measurable $Q'\subset Q\subset X\times Y$. 
\end{lemma}

For returns to rectangles, say $A_l\times B_l$, for fixed $y\in Y$, we can ignore all the times $j$ where $R^j(y)\not \in B_l$. Hence
denoting $\alpha_l(y)=\Phi_{B_l}(y)$, we first show that for $\nu$-a.e $y\in Y$
\begin{equation}\label{delaypoiscon:1}
\mu(A_l)\Sigma_{A_l,\alpha_l(y)}\xRightarrow{\mu} \Sigma_{Exp} \;\;\;\as{l}.
\end{equation}
To show this, the idea is that, due to assumption (BR), the times $\tilde{\alpha}_l^{(j)}(y)$ are sufficiently far apart to use (MEM), we apply Proposition \ref{delaykapretprop}. Since \eqref{delaypoiscon:1} is now true for $\nu$-a.e $y\in Y$, we also have
\begin{equation}\label{delaypoiscon:2}
\mu(A_l)\Sigma_{A_l,\Phi_{B_l}}\xRightarrow{\mu\times \nu} \Sigma_{Exp} \;\;\;\as{l},
\end{equation}
where $\Sigma_{A_l,\Phi_{B_l}}(x,y)=\Sigma_{A_l,\Phi_{B_l}(y),y}(x)$.\\
These are not quite the returns that we wanted to consider. However notice that, in every step, we skip exactly $\f{B_l}$ steps, thus we have the following relation
\begin{equation}\label{rectretsumrel}
\sigma_{A_l\times B_l}(x,y)=\sum_{j=0}^{\sigma_{A_l,\Phi_{B_l}(y),y}(x)-1} \f{B_l}\circ R^j_{B_l}(y).
\end{equation}
We now use (EE) to control the ergodic sums of $\varphi_{B_l}$ and show that
\begin{equation*}
\mu(A_l)\nu(B_l)\Sigma_{A_l\times B_l}\xRightarrow{\mu\times \nu} \Sigma_{Exp} \;\;\;\as{l}.
\end{equation*}
The same idea %essentially 
works for unions of rectangles $\DS \bigcup_{k=1}^K A_l^{(k)}\times B_l^{(k)}$.

\subsection{Quantitative Ergodic Theorem}

In section \ref{unifestsec}, it will be convenient to use a pointwise quantitative ergodic Theorem instead of the $L^2$ bound we assume in (EE). We will show first that such pointwise bounds hold for most points.
\\
More explicitly, we will show that for every sequence $(\s{f}{l})$ of functions in $L^2$ and $\epsilon>0$ there are $L_y,M_l\geq 1$ and a constant $K>0$ such that
\begin{equation*}
\left|\sum_{j=1}^n f_l \circ R^j(y) - n\int_Y f_l d\nu \right|\leq K ||f_l||_{C^{r'}} n^{\frac{1+2\delta+\epsilon}{3}} \;\;\;\textit{for $\nu$-a.e }y\in Y, \forall n\geq M_l, l\geq L_y.
\end{equation*} 
\\
First let us give a small Lemma.
\begin{lemma}\label{sparselem}
Let $0<\delta<1$, and $k_n =\left\lfloor n^{\delta'} \right\rfloor$, for some $\delta'>1$ with $1>\delta'(1-\delta)$, then for some big enough $M\geq 1$ we have
\begin{equation*}
[M,\infty)\subset \bigcup_{n\geq 1} [k_n-k_n^{\delta},k_n+k_n^{\delta}].
\end{equation*}
\end{lemma}
\begin{proof}
Since $k_n\rightarrow\infty$ as $n\rightarrow\infty$, it is enough to show
\begin{equation*}
k_n+k_n^{\delta}\geq k_{n+1}-k_{n+1}^{\delta},
\end{equation*}
for big enough $n$. Indeed, using the mean value Theorem there is $\xi_n\in [n,n+1]$ such that
\begin{align*}
(n+1)^{\delta'}-n^{\delta'}=\delta'\xi_n^{\delta'-1}\leq \delta'(n+1)^{\delta'-1}\leq (n+1)^{\delta\delta'},
\end{align*}
as soon as $n^{\delta\delta'-\delta'+1}>\delta'$.
\end{proof}

\begin{proposition}\label{l2boundprop}
Let $(Y,R,\nu)$ be a probability preserving ergodic dynamical system, $\s{f}{l}$ be a sequence of $C^{r'}$ functions. Suppose $R$ satisfies (EE), then, for small enough $\epsilon>0$, there are $L_y\geq 1$ and a constant $K>0$ such that
\begin{equation*}
\left|\sum_{j=1}^n f_l \circ R^j(y) - n\int_Y f_l d\nu \right|\leq K ||f_l||_{C^{r'}} n^{\delta_1} \;\;\;\textit{for $\nu$-a.e }y\in Y, \forall n\geq M_l, l\geq L_y,
\end{equation*}
where
\begin{equation*}
M_l=\left\lceil l^{\frac{-\delta_2}{2(\delta-\delta_1)\delta_2+1}+\delta_2\epsilon}\right\rceil ,
\end{equation*}
and $\delta_1=\frac{1+2\delta+\epsilon}{3}$ and $\delta_2=\frac{3}{2(1-\delta)}$.
\end{proposition}
\begin{proof}
(1)  W.l.o.g we may assume $\int_Y f_l \d\nu=0$. For $l\geq 1$, $n\geq 1$, $K>0$ and $\delta_1\in(0,1)$ close to $1$, using the Chebyshev inequality we have
$$
\nu\left( \left|\sum_{j=1}^n f_l\circ R^j\right| > K||f_l||_{C^{r'}} n^{\delta_1}\right) \leq \frac{\left|\left| \sum_{j=1}^n f_l\circ R^j\right|\right|_{L^2}^2}{(K||f_l||_{C^{r'}} n^{\delta_1})^2}
\leq \frac{C^2}{K^2} n^{2(\delta-\delta_1)}.
$$
Denote $B_n^{l,K}=\left\{\left|\sum_{j=1}^n f_l\circ R^j\right| > K||f_l||_{C^{r'}} n^{\delta_1}\right\}$. We will use a Borel-Cantelli argument to show that, for some $K>0$ and for $\nu$-a.e $y\in Y$, there are only finitely many $l\geq 1, n\geq M_l$ such that $y\in B_n^{l,K}$.

(2) Note that if 
\begin{equation*}
\left|\sum_{j=1}^n f_l\circ R^j(y)\right| > (2K+1)||f_l||_{C^{r'}} n^{\delta_1},
\end{equation*}
then
\begin{equation*}
\left|\sum_{j=1}^k f_l\circ R^j(y)\right| > K||f_l||_{C^{r'}} k^{\delta_1},
\end{equation*}
for $k\in \mathbb{Z}\cap [n-n^{\delta_1},n+n^{\delta_1}]$. By Lemma \ref{sparselem}, for any $\delta_2>1$ with $1>\delta_2(1-\delta_1)$ we have that there is some $M\geq 1$, only depending on $\delta_1$ and $\delta_2$, such that for each $k\geq M$ there is a $n \geq 1$ such that $B_k^{l,K}\subset B_{k_n}^{l,(2K+1)}$, where $k_n=\left\lfloor n^{\delta_2}\right\rfloor$. 
(3) For some small $\epsilon>0$ let 
\begin{equation*}
\tilde{M}_l=\left\lceil l^{\frac{-1}{2(\delta-\delta_1)\delta_2+1}+\epsilon}\right\rceil.
\end{equation*}
Suppose $2(\delta-\delta_1)\delta_2<-1$ then
\begin{align*}
\sum_{l\geq 1}\sum_{n\geq \tilde{M}_l} \nu(B_{k_n}^{l,(2K+1)})&\leq \sum_{l\geq 1}\sum_{n\geq \tilde{M}_l} \frac{C^2}{(2K+1)^2} n^{2(\delta-\delta_1)\delta_2}\\
&\leq \frac{C^2}{(2K+1)^2 2(\delta-\delta_1)\delta_2} \sum_{l\geq 1} \tilde{M}_l^{2(\delta-\delta_1)\delta_2+1}
&<\infty.
\end{align*}
By Borel-Cantelli, for $\nu$-a.e $y\in Y$ there are only finitely many $l\geq 1$ and $n\geq \tilde{M}_l$ such that $y\in B_{k_n}^{l,(2K+1)}$. It follows that, for large enough $L_y\geq 1$
\begin{equation*}
y\not\in B_n^{l,K}\;\;\; \forall l\geq L_y, n\geq M_l,
\end{equation*}
where 
\begin{equation*}
M_l=\left\lceil l^{\frac{-\delta_2}{2(\delta-\delta_1)\delta_2+1}+\delta_2 \epsilon}\right\rceil.
\end{equation*}

(4) For $\delta_1,\delta_2$ we required
\begin{itemize}
\item $\delta_1>0$,
\item $\delta_2>1$,
\item $(1-\delta_1)\delta_2<1$,
\item and $2(\delta-\delta_1)\delta_2<-1$.
\end{itemize}
For $\delta_1=\frac{1+2\delta+\epsilon}{3}$ and $\delta_2=\frac{3}{2(1-\delta)}$, these relations are satisfied.
\end{proof}

\begin{corollary}\label{pointeecor}
Suppose $R$ satisfies (EE), and let $\s{f}{l}$ be a sequence of functions in $C^{r'}$. Then for $\nu$-a.e $y\in Y$ and each $\epsilon>0$, there are $L_y\geq 1$, $\kappa>0$ and a constant $K>0$ such that
\begin{equation*}
\left|\sum_{j=1}^n f_l \circ R^j(y) - n\int_Y f_l d\nu \right|\leq K ||f_l||_{C^{r'}} n^{\delta_1} \;\;\;\textit{for $\nu$-a.e }y\in Y, \forall n\geq l^{\frac{\delta}{1-\delta}}, l\geq L_y,
\end{equation*}
where
\begin{equation*}
\delta_1=\frac{2+\delta}{3}.
\end{equation*}
\end{corollary}
\begin{proof}
Apply Proposition \ref{l2boundprop}, we have
\begin{align*}
\frac{-\delta_2}{2(\delta-\delta_1)\delta_2+1}+\delta_2\epsilon=\frac{1-\epsilon}{2(1-\delta+\epsilon)-2(1-\delta)}=\frac{1-\epsilon}{\epsilon},
\end{align*}
set $\epsilon=1-\delta$.
\end{proof}

\subsection{Uniform Estimates}\label{unifestsec}

Let $\s{B}{l}$ be a sequence of rare events\footnote{Later on in \S \ref{scalePLTsec} we will take finitely many such sequences $\s{B^{(1)}}{l},...\s{B^{(K)}}{l}$, but the same arguments apply.} in $Y$, such that $\s{B}{l}$ is regularly approximable in $C^{r'}$. The goal of this section will be to show
\begin{equation}
\sup_{l\geq 1}\left|\frac{\nu(B_l)}{sN_l}\sum_{j=1}^{\lceil sN_l\rceil} \f{B_l}^{(j)} (y)-1 \right|\rightarrow 0 \;\;\;\as{s},\; \textit{$\nu$-a.e } y\in Y,
\end{equation}
for some $N_l>0$.

\begin{lemma}\label{indcrlem}
Let $R$ satisfy (EE) and let $\s{B}{l}$ be a sequence of rare events in $Y$, such that $\s{B}{l}$ is regularly approximable in $C^{r'}$. Then there is a constant $K'>0$ only depending on and $\app{\s{B}{l}}$, and, for $\nu$-a.e $y\in Y$, $M_y>0$ only depending on $y$ and $\app(\s{B}{l})$ such that
\begin{equation}
\label{IndicatorSum}
\left|\sum_{j=1}^n 1_{B_l}\circ R^j(y)-n\nu(B_l)\right|\leq K' n^{1-\frac{1-\delta_1}{r'+1}}, \;\;\;\forall n\geq M_yM_l, l\geq 1 \textit{ $\nu$-a.e }y\in Y,
\end{equation}
where $M_l=\nu(B_l)^{-\frac{r'+1}{d'(1-\delta_1)}}$.
\end{lemma}
\begin{proof}
Let $C=\app(\s{B}{l})$, then for $k>\nu(B_l)^{-\frac{1}{d'}}$ there are $\underline{h}_{k,l},\overline{h}_{k,l}\in C^{r'}$ with $\underline{h}_{k,l}\leq 1_{B_l}\leq \overline{h}_{k,l}$ and
\begin{equation*}
\nu(1_{B_l}\not= h_{k,l})\leq\nu(B_l)^{\frac{d'-1}{d'}}\frac{1}{k}, \;\textit{ while }\; ||h_{k,l}||_{C^{r'}}\leq Ck^{r'},
\end{equation*}
for $h_{k,l}\in \{\overline{h}_{k,l},\underline{h}_{k,l}\}$. In particular
\begin{equation*}
\nu(B_l)-\nu(B_l)^{\frac{d'-1}{d'}}\frac{1}{k} \leq \int_Y \underline{h}_{k,l} \d\nu \leq \int_Y\overline{h}_{k,l}\d\nu \leq \nu(B_l)+\nu(B_l)^{\frac{d-1}{d}}\frac{1}{k}.
\end{equation*}
By Corollary \ref{pointeecor} there are $K>0$ and $I_y\geq1$ such that, for $h_{k,l}\in \{\overline{h}_{k,l},\underline{h}_{k,l}\}$, we have
\begin{equation*}
\left|\sum_{j=1}^n h_{k,l}\circ R^j(y) - n\int_Y h_{k,l}\d\nu\right|\leq K||h_{k,l}||_{C^{r'}}n^{\delta_1} \;\;\;\forall n\geq (l+k)^{\frac{\delta}{1-\delta}}, l+k\geq I_y \textit{ $\nu$-a.e }y\in Y.
\end{equation*} 
Therefore, for such $k,l,n$,
\begin{align*}
\sum_{j=1}^n 1_{B_l} \circ R^j(y) &\geq \sum_{j=1}^n \underline{h}_{k,l}\circ R^j(y)
\geq n\int_Y \underline{h}_{k,l} \d\nu - KCk^rn^{\delta_1}\\
&\geq n\nu(B_l)-n\nu(B_l)^{\frac{d'-1}{d'}}\frac{1}{k}- KCk^{r'}n^{\delta_1}.
\end{align*}
Likewise
\begin{equation}\label{sumibeq:1}
\sum_{j=1}^n 1_{B_l} \circ R^j(y)\leq n\nu(B_l)+n\nu(B_l)^{\frac{d'-1}{d'}}\frac{1}{k}+ KCk^{r'}n^{\delta_1}.
\end{equation}
Hence
\begin{equation*}
\left|\sum_{j=1}^n 1_{B_l} \circ R^j(y)-n\nu(B_l)\right|\leq n\nu(B_l)^{\frac{d'-1}{d'}}\frac{1}{k}+ KCk^{r'}n^{\delta_1}.
\end{equation*}
Let $k_n=\left\lceil n^{\frac{1-\delta_1}{r'+1}} \right\rceil$ and 
\begin{equation*}
M_y=(100I_y)^{\frac{r'+1}{1-\delta_1}}, \; \textit{ and } \; M_l=\nu(B_l)^{-\frac{r'+1}{(1-\delta_1)d'}}.
\end{equation*}
We can apply \eqref{sumibeq:1} for $n\geq  M_yM_l$, $k_n$ and $l\geq 1$ (since $k_n>I_y$ and\footnote{To be completely correct, one would have to consider $k_n+l$ instead of $l$, but by renumbering we can assume that $l$ is very small compared to $\nu(B_l)^{-1}$.} $k_n^{\frac{\delta}{\delta-1}}<n$) to obtain
\begin{align*}
\left|\sum_{j=1}^n 1_{B_l} \circ R^j(y)-n\nu(B_l)\right|&\leq K' n^{1-\frac{1-\delta_1}{r'+1}},
\end{align*}
where the constant $K'>0$ only depends on $C$.
\end{proof}

\begin{proposition}\label{UCprop:1}
Suppose $R$ satisfies (EE) and let $\s{B}{l}$ be a sequence of rare events in $Y$, such that $\s{B}{l}$ is regularly approximable in $C^{r'}$. Then, for $\nu$-a.e $y\in Y$, there are $M_y>0$ such that, for $s\geq M_y$, we have
\begin{equation}\label{phisumdeceq}
\left|\frac{\nu(B_l)}{sN_l}\sum_{j=0}^{\lceil sN_l\rceil-1} \f{B_l}\circ R^j_{B_l}(y) -1 \right|\leq  4s^{-\frac{1-\delta_1}{r'+1}}.
\end{equation}
where
\begin{equation*}
N_l= \left\lceil K \nu(B_l)^{\frac{r'+1}{1-\delta_1}} \right \rceil,
\end{equation*}
for some constant $K>0$.
\end{proposition}
\begin{proof}
Let $y\in Y$ be as in the conclusion of Lemma \ref{indcrlem}. To simplify the notation denote $\DS g(n)=n^{1-\frac{1-\delta_1}{r'+1}}$. 
By \eqref{IndicatorSum}, for $l\geq 1$ and $n\geq M_yM_l$ we have 
\begin{equation*}
n\nu(B_l)-K' g(n) \leq \sum_{j=1}^n 1_{B_l} \circ R^j_{B_l}(y) \leq n\nu(B_l) + K' g(n).
\end{equation*}
and thus
\begin{align*}
\sum_{j=0}^{\lfloor n\nu(B_l)-K'g(n) \rfloor-1} & \f{B_l}\circ R^j_{B_l}(y) \leq n \\
&\leq \sum_{j=0}^{\lceil n\nu(B_l)+K'g(n) \rceil-1} \f{B_l}\circ R^j_{B_l}(y).
\end{align*}
Rewrite the summation limits above as 
$\DS \left\lfloor n\left(\nu(B_l)-\frac{K'g(n)}{n}\right) \right\rfloor-1$
and $\DS
\left\lceil n\left(\nu(B_l)+\frac{K' g(n)}{n}\right) \right\rceil-1$ respectively.
For $N_l$ as in the statement and $s\geq M_y$ set 
\begin{equation*}
n=\fceil{sN_l}{\nu(B_l)-\frac{g(\lceil sN_l\rceil)}{\lceil sN_l\rceil}}.
\end{equation*}
By definition of $N_l$, and since $s\geq 1$ we have
\begin{equation*}
\frac{g(\lceil sN_l\rceil)}{\lceil sN_l\rceil} \leq \frac{1}{2K' }\nu(B_l),
\end{equation*}
and therefore
\begin{equation}\label{snlyeq:1}
n\geq\fceil{sN_l}{\nu(B_l)\left(1-\frac{1}{2K' }\right)}\geq \lceil s N_l \rceil.
\end{equation}
It follows that also
\begin{equation*}
\frac{g(n)}{n}\leq \frac{1}{2K' }\nu(B_l),
\end{equation*}
thus, using (both estimates in) \eqref{snlyeq:1}, we obtain
\begin{align*}
\left\lfloor n\left(\nu(B_l)-\frac{K' g(n)}{n}\right) \right\rfloor & \geq n \frac{\nu(B_l)}{2} \\
&\geq \lceil s N_l \rceil.
\end{align*}
So
\begin{align*}
\sum_{j=0}^{\lceil sN_l\rceil-1} \f{B_l}\circ R^j_{B_l}(y)& \leq \sum_{j=0}^{\left\lfloor n\left(\nu(B_l)-\frac{K' g(n)}{n}\right) \right\rfloor-1} \f{B_l}\circ R^j_{B_l}(y)\\
&\leq n \leq \frac{sN_l}{\nu(B_l)-\frac{g(\lceil sN_l\rceil)}{\lceil sN_l\rceil}}+1
\end{align*}
and thus
\begin{align*}
\frac{\nu(B_l)}{sN_l}\sum_{j=0}^{\lceil sN_l\rceil-1} \f{B_l}\circ R^j_{B_l}(y) & \leq \frac{\nu(B_l)}{\nu(B_l)-\frac{g(\lceil sN_l\rceil)}{\lceil sN_l\rceil}}+\frac{2}{s}\\
&\leq 1+\frac{\frac{g(\lceil sN_l\rceil)}{\lceil sN_l\rceil}}{\nu(B_l)-\frac{g(\lceil sN_l\rceil)}{\lceil sN_l\rceil}}+\frac{2}{s}\\
&\leq 1+\frac{\frac{g(\lceil sN_l\rceil)}{\lceil sN_l\rceil}}{\nu(B_l)\left(1-\frac{1}{2K' }\right)}+\frac{2}{s}\\
&\leq 1+\frac{\frac{g(\lceil sN_l\rceil)}{\lceil sN_l\rceil}}{2K' \frac{g(\lceil N_l\rceil)}{\lceil N_l\rceil}\left(1-\frac{1}{2K' }\right)}+\frac{2}{s}\\
&\leq 1+\frac{g(\lceil s\rceil)}{2K'(s+1)}+\frac{2}{s}.
\end{align*}
Analogously 
$\displaystyle
\frac{\nu(B_l)}{sN_l}\sum_{j=0}^{\lceil sN_l\rceil-1} \f{B_l}\circ R^j_{B_l}(y) \geq  1-\frac{g(\lceil s\rceil)}{2K'(s+1)}-\frac{2}{s}.
$
The claim \eqref{phisumdeceq} follows since we can take $K'>1$, and choosing $K>0$ such that
\begin{equation*}
N_l=\max\left(M_l, \fceil{1}{2\nu(B_l)}, \min\left(n\geq 1\;|\; n^{-\frac{1-\delta_1}{r'+1}}\leq \frac{1}{2K'}\nu(B_l)\right)\right)= \left\lceil K \nu(B_l)^{\frac{r'+1}{1-\delta_1}} \right \rceil.
\end{equation*}
\end{proof}

This also yields the desired uniform decay.
\begin{corollary}\label{UCcor}
Suppose $R$ satisfies (EE) and let $\s{B}{l}$ be a sequence of rare events in $Y$, such that $\s{B}{l}$ is regularly approximable in $C^{r'}$. Then there is a constant $K>0$ such that
\begin{equation*}
\sup_{l\geq 1}\left|\frac{\nu(B_l)}{sN_l}\sum_{j=1}^{\lceil sN_l\rceil} \f{B_l}^{(j)} (y)-1 \right|\rightarrow 0 \;\;\;\as{s},\; \textit{$\nu$-a.e } y\in Y,
\end{equation*}
where
\begin{equation*}
N_l=\left \lceil K \nu(B_l)^{-\frac{r'+1}{1-\delta_1}} \right \rceil = \left \lceil K \nu(B_l)^{-3\frac{r'+1}{1-\delta}} \right \rceil
\end{equation*}
\end{corollary}

\subsection{PLT scaled by returns to $\{B_l\}$}\label{scalePLTsec}

For the rest of this exposition let $K\geq 1$, and $A_l^{(1)},...,A_l^{(K)}$ resp. $B_l^{(1)},...,B_l^{(K)}$ be subsets of $X$ resp. $Y$ regularly approximable in $C^r$ resp. $C^{r'}$. Suppose that there are $r_l\searrow 0$, and positive constants $\omega^{(k)},\theta^{(k)}>0$ such that
\begin{equation*}
\mu(A_l^{(k)})=\omega^{(k)} r_l^d+o(r_l^d),\;\textit{ and } \;\nu(B_l^{(k)})=\theta^{(k)}r_l^{d'}+o(r_l^{d'}) \;\;\;\forall l\geq 1, k=1,...,K,
\end{equation*}
and, for each $l$, $B_l^{(1)},...,B_l^{(K)}$ are disjoint. Denote $\DS B_l=\bigcup_{k=1}^K B_l^{(k)}$, which, by disjointness of $B_l^{(1)},...,B_l^{(K)}$, is also regularly approximable in $C^{r'}$, and $\alpha_l=\Phi_{B_l}$. 
Consider 
\begin{equation*}
\kappa_l^{(n)}(y)=k \;\;\; \textit{ if } \;\; R_{B_l}^n(y)\in B_l^{(k)},
\end{equation*}
which is well-defined as the $B_l^{(k)}$ are disjoint, and
\begin{align*}
&\sigma_{\kappa_l,\alpha_l,y}^{(1)}(x)
=\min\left(n\geq 1\;|\; T_y^{\tilde{\alpha}_l^{(n)})}(x)\in A_l^{(\kappa_l^{(n)})}\right),\\
&\sigma_{\kappa_l,\alpha_l,y}^{(n+1)}(x)=
\min\left(k\geq \sigma_{\kappa_l,\alpha_l,y}^{(n)}(x)+1\;|\; T_y^{\tilde{\alpha}_l^{(k)}}(x)\in A_l^{(\kappa_l^{(k)})}\right),\\
&\Sigma_{\kappa_l,\alpha_l,y}=(\sigma_{\kappa_l,\alpha_l,y}^{(1)},\sigma_{\kappa_l,\alpha_l,y}^{(2)},...).
\end{align*}
Following the steps outlined in \S \ref{secThestrategy} we will first show
\begin{equation}\label{Phikapdelcon}
\Omega r_l^{d'} \Sigma_{\kappa_l,\alpha_l,y}\xRightarrow{\mu} \Sigma_{Exp} \;\;\;\as{l}, \; \nu-a.a \;y\in Y
\end{equation}
for $\Omega=\frac{\sum_{k=1}^K \omega^{(k)} \theta^{(k)}}{\sum_{k=1}^K \theta^{(k)}}$.
Then use the relation
\begin{equation}\label{unionrectretsumrel:2}
\sigma_{\bigcup_{k=1}^K A_l^{(k)} \times B_l^{(k)}}(x,y)=\sum_{j=0}^{\sigma_{\kappa_l,\Phi_{B_l}(y)}(x)-1} \f{B_l}\circ R^j_{B_l}(y)
\end{equation}
to obtain
\begin{equation}\label{poisconunionrect}
(\mu\times \nu)\left(\bigcup_{k=1}^K A_l^{(k)} \times B_l^{(k)} \right)\Sigma_{\bigcup_{k=1}^K A_l^{(k)} \times B_l^{(k)}}\xRightarrow{\mu\times \nu} \Sigma_{Exp} \;\;\;\as{l}.
\end{equation}
Denote
\begin{equation*}
p_{l,t}^{(k)}(y):=\frac{1}{t}\#\{j=1,...,t_l\;|\;\kappa_l^{(j)}(y) =k\}=\frac{1}{t}\sum_{j=1}^{t} 1_{B_l^{(k)}}(R_{B_l}^j(y)).
\end{equation*}
We first show that, for each $k$, $\nu$-a.e $y\in Y$, and $t_l=O(r_l^{-d})$, we have
\begin{equation}\label{kapprobcon:2}
p_{l,t_l}^{(k)}(y) \rightarrow \frac{\theta^{(k)}}{\sum_{j=1}^K \theta^{(j)}}=:p^{(k)}\;\;\; \as{l}.
\end{equation}

\begin{lemma}\label{kaplem}
Suppose $R$ satisfies (EE) and 
\begin{equation*}
d> 3d' \frac{r'+1}{1-\delta}
\end{equation*}
Then $\kappa_l$ satisfies \eqref{kapprobcon:2}.
\end{lemma}
\begin{proof}
Since $\delta_1=\frac{2+\delta}{3}$, the assumption is equivalent to
\begin{equation*}
d>d'\frac{r'+1}{1-\delta_1}.
\end{equation*}
(i) Using Lemma \ref{indcrlem}, and disjointness we obtain a constant $K'>0$ and $M_y\geq 1$ such that, for $\nu$-a.e $y\in Y$, we have
\begin{align*}
&\left|\sum_{j=1}^n 1_{B_l^{(k)}}\circ R^j(y)-n\nu(B_l^{(k)})\right|\leq K'n^{1-\frac{1-\delta_1}{r'+1}}, \;\textit{ and}\\
&\left|\sum_{j=1}^n 1_{B_l}\circ R^j(y)-n\nu(B_l)\right|\leq K'n^{1-\frac{1-\delta_1}{r'+1}}
\end{align*}
for all $n\geq M_y$ and $k=1,...,K$. On the other hand\footnote{Making $K'$ and $M_y$ bigger if needed.}, for $s>M_y$, Proposition \ref{UCprop:1} yields
\begin{align*}
\left|\frac{\nu(B_l)}{sN_l}\sum_{j=0}^{\lceil sN_l\rceil-1} \f{B_l}\circ R^j_{B_l}(y) -1 \right|\leq  4s^{-\frac{1-\delta_1}{r'+1}}.
\end{align*}
where
\begin{equation*}
N_l= \left\lceil K' \nu(B_l)^{-\frac{r'+1}{1-\delta_1}} \right \rceil.
\end{equation*}
(ii) Denote $a=\frac{r'+1}{1-\delta_1}$. Rewrite
\begin{equation*}
p_{l,n}^{(k)}=\frac{1}{n}\sum_{i=1}^{\sum_{j=1}^n \f{B_l}\circ R_{B_l}^j(y)} 1_{B_l^{(k)}} \circ R^i.
\end{equation*} 
Denote $t_l=s_lN_l$, and let $l$ be big enough so that $s_l>M_y$, then we have 
\begin{align*}
t_lp_{l,t_l}^{(k)}(y)&=\sum_{j=1}^{\sum_{j=1}^{t_l} \f{B_l}\circ R_{B_l}^j-1} 1_{B_l^{(k)}}\circ R^j(y)\leq \sum_{j=1}^{t_l \nu(B_l)^{-1}(1+4s^{-a})-1} 1_{B_l^{(k)}}\circ R^j(y)\\
&\leq t_l \nu(B_l)^{-1}\nu(B_l^{(k)})(1+4s^{-a}) + K' \left(t_l \nu(B_l)^{-1}(1+4s^{-a})\right)^{1-\frac{1}{a}}\\
&\leq t_l\frac{\nu(B_l^{(k)})}{\nu(B_l)}(1+(4+K')s_l^{-a}).
\end{align*}
The upper bound is similar. So
\begin{equation}\label{pesteq}
\left|p_{l,n}^{(k)}(y)-\frac{\nu(B_l^{(k)})}{\nu(B_l)}\right|=O(s_l^{-a}).
\end{equation}

Since 
\begin{equation*}
N_l=O(\nu(B_l)^{-\frac{r'+1}{1-\delta_1}})=O(r_l^{-d'\frac{r'+1}{1-\delta_1}})=o(r_l^d),
\end{equation*}
we necessarily have $s_l\rightarrow \infty$, and \eqref{kapprobcon:2} follows from \eqref{pesteq}. 

\end{proof}

\subsection{Adding in the gaps}

Now all that's left to do is to add back in the gaps. As mentioned in \S \ref{scalePLTsec}, having shown \eqref{Phikapdelcon}, we will now explain how to conclude
\begin{equation*}
(\mu\times \nu)\left(\bigcup_{k=1}^K A_l^{(k)}\times B_l^{(k)}\right) \Sigma_{\bigcup_{k=1}^K A_l^{(k)}\times B_l^{(k)}} \xRightarrow{\mu\times\nu} \Sigma_{Exp}, \;\;\;\as{l}.
\end{equation*}
using the relation
\begin{equation*}
\sigma_{\bigcup_{k=1}^K A_l^{(k)}\times B_l^{(k)}}=\sum_{j=0}^{\sigma_{\kappa_l,\alpha_l,y}(x)-1} \f{B_l}\circ R^j_{B_l}(y).
\end{equation*}

This is rather straightforward, given Corollary \ref{UCcor}, and follows from a more general principal in probability theory. As this principle finds use in various places and has, to the authors knowledge, not been formulated in generality, let us state and prove a more general version than we need here.

\begin{lemma}\label{weakzeroconlem}
Let $(\Omega,\mathbb{P})$ be a probability space, and $E_l:\Omega\rightarrow[0,\infty)$ non-negative real random variables, such that there are positive random variables $\mu_l:\Omega\rightarrow (0,1)$ with
\begin{equation*}
\mu_l E_l\xRightarrow{\mathbb{P}}E\;\;\; \as{l},
\end{equation*}
for some non-negative random variable $E$ with $\mathbb{P}(E=0)=0$. Then for any $N_l:\Omega\rightarrow[0,\infty)$ with
\begin{equation*}
\mu_l N_l\rightarrow 0 \;\;\;\as{l} \; \textit{ pointwise $\mathbb{P}$-a.e}
\end{equation*}
we have 
\begin{equation*}
\mathbb{P}(E_l\leq N_l) \rightarrow 0 \;\;\;\as{l}.
\end{equation*}
\end{lemma}
\begin{proof}
Let $\epsilon>0$, there is a $\delta>0$ such that $\mathbb{P}(E\leq \delta)<\epsilon$, and the distribution function of $E$ is continuous at $\delta$. By Jegorow's Theorem, there is a measurable $K\subset\Omega$ with $\mathbb{P}(K^c)<\epsilon$ and 
\begin{equation*}
\mu_l N_l \rightarrow0 \;\;\;\as{l} \;\textit{ uniformly on }K.
\end{equation*}
Choose $\tilde{L}\geq 1$ so big that $N_l\leq \delta \mu_l^{-1}$ on $K$ for $l\geq \tilde{L}$. Now choose $L\geq\tilde{L}$ so big that
\begin{equation*}
\mathbb{P}(\mu_l E_l \leq \delta)\leq 2\epsilon\;\;\;\forall l\geq L,
\end{equation*}
it follows that
\begin{equation*}
\mathbb{P}(E_l\leq N_l)\leq \mathbb{P}(\mu_lE_l\leq \delta)+ \mathbb{P}(K^c)\leq 3 \epsilon,
\end{equation*}
for $l\geq L$.
\end{proof}

\begin{proposition}\label{sumweakconprop}
Let $(\Omega,\mathbb{P})$ be a probability space, and $E_l:\Omega \rightarrow \mathbb{N}$ be positive integer valued observables. Assume there are positive real numbers $q_l\searrow0$, 
and a $[0,\infty)$-valued random variable $E$ with $\mathbb{P}(E=0)=0$
such that
\begin{equation*}
q_l E_l\xRightarrow{\mathbb{P}} E \;\;\;\as{l},
\end{equation*}

Let $\alpha^{(l)}_j:\Omega \rightarrow [0,\infty)$ be non-negative random variables, and assume there are 
$N_l:\Omega\rightarrow (0,\infty)$ with $q_l N_l\rightarrow 0$ as $l\rightarrow\infty$ $\mathbb{P}$-a.e and
positive random variables $b_l:\Omega \rightarrow (0,\infty)$
such that
\begin{equation*}\tag{UC}
\sup_{l\geq 1}\left|\frac{1}{sN_l b_l}\sum_{j=1}^{\lceil sN_l\rceil} \alpha_l^{(j)} (\omega)-1 \right|\rightarrow 0 \;\;\;\as{s},\; \textit{ $\mathbb{P}$-a.e. } .
\end{equation*}

Then 
\begin{equation*}
\frac{q_l}{b_l}\sum_{j=1}^{E_l} \alpha_l^{(j)} \xRightarrow{\mathbb{P}} E\;\;\;\as{l}.
\end{equation*}
\end{proposition}

\begin{remark}\label{sumweakconrem}
In our context we use $(\Omega,\mathbb{P})=(X\times Y,\mu\times \nu)$, $R_l=\sigma_{A_l,\Phi_{B_l}}$, $\alpha_l^{(j)}=\f{B_l}\circ R_{B_l}^j$, $q_l=\mu(A_l)$ and $b_l=\frac{1}{\nu(B_l)}$. The existence of $N_l$ is the content of Corollary \ref{UCcor}.
\end{remark}

\begin{proof}
(i) Let $F$ be the distribution function of $E$ and $C=\{t\;|\; F\textit{ is continuous at }t\}$ the set of its continuities. Let $t\in C$, and $\epsilon>0$ such that $\frac{t}{1+\epsilon},\frac{t}{1-\epsilon}\in C$.
\\
\noindent
(ii) By Jegorow's Theorem, for $l\geq 1$, there is a measurable set $K\subset\Omega$ with $\mathbb{P}(K^c)<\epsilon$
\begin{equation*}
\sup_{l\geq 1}\sup_{\omega\in K} \left|\frac{1}{sN_l b_l}\sum_{j=1}^{\lceil sN_l\rceil} \alpha_l^{(j)} (\omega)-1 \right|\rightarrow 0 \;\;\;\as{s}.
\end{equation*}
\\
(iii)
Choose $S>0$ so big that 
\begin{equation*}
(1-\epsilon) b_l\leq \frac{1}{sN_l}\sum_{j=1}^{\lceil sN_l\rceil} \alpha_l^{(j)} \leq (1+\epsilon) b_l\;\;\;\textit{on }K,\; \forall s\geq S, l\geq 1.
\end{equation*}
Then
\begin{equation*}
(1-\epsilon) q_l E_l\leq \frac{q_l}{b_l} \sum_{j=1}^{E_l} \alpha_l^{(j)} (\omega)\leq (1+\epsilon) q_l E_l\;\;\;\textit{ on }K\cap \{E_l\geq SN_l\},\; \forall l\geq 1.
\end{equation*}
We get 
\begin{align*}
&\mathbb{P}\left(\frac{q_l}{b_l}  \sum_{j=1}^{E_l} \alpha_l^{(j)} \leq t\right)  -\mathbb{P}(E_l\leq SN_l)-\mathbb{P}(K^c)\\
&\leq \mathbb{P}\left( K\cap \{E_l\geq SN_l\} \cap \left\{ \frac{q_l}{b_l}  \sum_{j=1}^{E_l} \alpha_l^{(j)} \leq t\right\} \right)
 \leq \mathbb{P}\left( q_l E_l \leq \frac{t}{1-\epsilon} \right).
\end{align*}
Likewise
\begin{equation}\label{LowerB}
\begin{aligned}
\mathbb{P}\left(\frac{q_l}{b_l}  \sum_{j=1}^{E_l} \alpha_l^{(j)} \leq t\right)
& \geq \mathbb{P}\left(K\cap \{E_l\geq SN_l\} \cap \left\{ q_l E_l \leq \frac{t}{1+\epsilon}\right\} \right) \notag \\
&\geq \mathbb{P}\left(q_l E_l \leq \frac{t}{1+\epsilon} \right)-\mathbb{P}(E_l\leq SN_l)-\mathbb{P}(K^c).
\end{aligned}
\end{equation}

Taking $\DS \limsup_{l\rightarrow\infty}$ resp. $\DS \liminf_{l\rightarrow\infty}$ in \eqref{LowerB}
and using Lemma \ref{weakzeroconlem} we obtain 
\begin{align*}
&F\left(\frac{t}{1+\epsilon}\right)-\epsilon\leq \liminf_{l\rightarrow \infty} \mathbb{P}\left(\frac{q_l}{b_l} \sum_{j=1}^{E_l} \alpha_l^{(j)} \leq t\right)\\
&\leq \limsup_{l\rightarrow \infty} \mathbb{P}\left(\frac{q_l}{b_l} \sum_{j=1}^{E_l} \alpha_l^{(j)} \leq t\right)
\leq F\left(\frac{t}{1-\epsilon}\right)+\epsilon.
\end{align*}

Since $C\subset (0,\infty)$ is dense, we can let $\epsilon\searrow 0$ while $\frac{t}{1+\epsilon},\frac{t}{1-\epsilon}\in C$, this yields
$\DS
\mathbb{P}\left(\frac{q_l}{b_l} \sum_{j=1}^{E_l} \alpha_l^{(j)} \leq t\right)\rightarrow F(t)\;\;\;\as{l}.
$
\end{proof}

\begin{remark}\label{weaksumconrem}
\begin{enumerate}[(i)]
\item We can extend this statement to sequences in the following manner: under the assumptions of the proposition, let $E^{(n)}_l:\Omega \rightarrow \mathbb{N}$ be such that
\begin{equation*}
q_l(E^{(1)}_l,E^{(2)}_l,...) \xRightarrow{\mu} (E^{(1)},E^{(2)},...) \;\;\;\as{l},
\end{equation*}
for some $E^{(n)}:\Omega\rightarrow[0,\infty)$ with $\mathbb{P}(E^{(n)}=0)=0$. Then
\begin{equation*}
\frac{q_l}{b_l}\left(\sum_{j=1}^{E^{(n)}_l} \alpha_l^{(j)}, \sum_{j=1}^{E^{(n)}_l} \alpha_l^{(j)},...\right) \xRightarrow{\mathbb{P}} (E^{(1)},E^{(2)},...) \;\;\;\as{l}.
\end{equation*}
The proof of this statement is almost the same as for the proposition, therefore we won't repeat it.
\item The probability measure $\mathbb{P}$ can be replaced by a sequence $\s{\mathbb{P}}{l}$ by also replacing $(UC)$ with
\begin{align*}
\forall \epsilon>0 \textit{ there is are} & \textit{ measurable sets }  K_l\subset \Omega \textit{ with } \limsup_{l\rightarrow\infty} \mathbb{P}_l(K_l^c)<\epsilon \;\textit{ such that}\\
&\sup_{l\geq 1}\sup_{\omega\in K_l} \left|\frac{1}{sN_l b_l}\sum_{j=1}^{\lceil sN_l\rceil} \alpha_l^{(j)} (\omega)-1 \right|\rightarrow 0 \;\;\;\as{s}.
\end{align*}
\end{enumerate}
\end{remark}

\begin{proof}[Proof of Theorem $\ref{plt:1}$]
(i) Let $r_l\searrow 0$, and denote by $Q_l=B_{r_l}(x^*,y^*)$ the geodesic ball of radius $r_l$ centred at $(x^*,y^*)$, and let $\epsilon>0$. Wlog $r_1$ is small enough that the exponential map at $(x^*,y^*)$ is a diffeomorphism from the ball of radius $2r_1$ in $\mathbb{R}^{d+d'}$ onto $B_{2r_1}(x^*,y^*)$. Let $\epsilon>0$, it is easy to construct, for some $K\geq 1$, sets $A_l^{(1)},...,A_l^{(K)}$ and $B_l^{(1)},...,B_l^{(K)}$ as in \S \ref{PLTsubssec}. Let $\omega^{(k)},\theta^{(k)}>0$ be such that\footnote{Choose $x^*,y^*$ such that the densities of $\mu,\nu$ are positive at the respective points.}
\begin{equation*}
\mu(A_l^{(k)})=\omega^{(k)} r_l^d+o(r_l^d),\;\textit{ and } \;\nu(B_l^{(k)})=\theta^{(k)}r_l^{d'}+o(r_l^{d'}) \;\;\;\forall l\geq 1, k=1,...,K,
\end{equation*}
and set $\Omega=\frac{\sum_{k=1}^K \omega^{(k)} \theta^{(k)}}{\sum_{k=1}^K \theta^{(k)}}$.

Due to Lemma \ref{approxlem:2} all those sets can be chosen to be regularly approximable, such that\footnote{Here we again use the continuity of the density.} $(\mu\times\nu)_{Q_l}\left(Q_l\setminus \bigcup_{k=1}^K A_l^{(k)} \times B_l^{(k)}\right) <\epsilon$ for all $l\geq 1$. Let $\Lambda>0$ be such that $(\mu\times\nu)(Q_l)=\Lambda r_l^{d+d'} +o(r_l^{d+d'})$, then
\begin{equation*}
\left|\Lambda-\Omega\sum_{k=1}^K \theta^{(k)}\right|<\epsilon
\end{equation*}
(ii) 
Denote $B_l=\bigcup_{k=1}^K B_l^{(k)}$, and, for $y$ as in assumption (MEM), consider $\alpha_l(y)=\Phi_{B_l}(y)$ and
\begin{equation*}
\kappa_l^{(n)}(y)=k \;\;\; \textit{ if } \;\; R_{B_l}^n(y)\in B_l^{(k)},
\end{equation*}
by disjointness $\kappa_l(y)$ is well-defined. By Lemma \ref{kaplem}, $\kappa_l(y)$ satisfies \eqref{kapprobcon:2}, and $p^{(k)}=\frac{\theta^{(k)}}{\sum_{s=1}^K \theta^{(s)}}$. We can 
use Proposition \ref{delaykapretprop} and Remark \ref{delaykapretrem}(i) to obtain
\begin{equation*}
\Omega r_l^d \Sigma_{\kappa_l(y),\alpha_l(y),y}\xRightarrow{\mu} \Sigma_{Exp} \;\;\;\as{l}, \textit{ $\nu$-a.e } y.
\end{equation*}
Since the convergence holds for $\nu$-a.e $y\in Y$, it follows that
\begin{equation*}
\Omega r_l^d \Sigma_{\kappa_l,\alpha_l}\xRightarrow{\mu\times\nu} \Sigma_{Exp} \;\;\;\as{l},
\end{equation*}
where $\Sigma_{\kappa_l,\alpha_l}(x,y)=\Sigma_{\kappa_l(y),\alpha_l(y),y}(x)$.
\\
\\
(iii) 
By Corollary \ref{UCcor}, $\alpha_l$ satisfies $(UC)$ with 
\begin{equation*}
b_l=\frac{1}{\nu(B_l)}=\frac{1}{\sum_{k=1}^K \theta^{(k)}}r_l^{-d'}+o(r_l^{-d'})
\end{equation*} 
and
\begin{equation*}
N_l=\lceil K' \nu(B_l)^{-3\frac{r'+1}{1-\delta}} \rceil,
\end{equation*}
making $K'$ bigger if necessary. Note that, for $Q_l'=\bigcup_{k=1}^K A_l^{(k)}\times B_l^{(k)}$ we have
\begin{align*}
\Sigma_{Q_l'}=\left(\sum_{j=0}^{\sigma_{\kappa_l,\alpha_l}^{(1)}-1} \f{B_l}\circ R^j_{B_l} ,\sum_{j=\sigma_{\kappa_l,\alpha_l}^{(1)}}^{\sigma_{\kappa_l,\alpha_l}^{(2)}-1} \f{B_l}\circ R^j_{B_l},...\right).
\end{align*}
We can apply Proposition \ref{sumweakconprop} resp. Remark \ref{weaksumconrem}(i)
\begin{equation*}
\Omega \left(\sum_{k=1}^K \theta^{(k)}\right) r_l^{d+d'}\Sigma_{Q_l'}\xRightarrow{\mu\times \nu}\Sigma_{Exp} \;\;\;\as{l}.
\end{equation*}
By disjointness 
\begin{equation*}
(\mu\times\nu)(Q'_l)=\Omega \left(\sum_{k=1}^K \theta^{(k)}\right) r_l^{d+d'} + o(r_l^{d+d'})
\end{equation*}
hence 
\begin{equation*}
(\mu\times\nu)(Q'_l)\Sigma_{Q_l'}\xRightarrow{\mu\times \nu}\Sigma_{Exp} \;\;\;\as{l}.
\end{equation*}
By the equivalence \eqref{phisigmaeqpuiv}, we have
\begin{equation*}
(\mu\times\nu)(Q'_l)\Phi_{Q_l'}\xRightarrow{\mu\times \nu}\Phi_{Exp} \;\;\;\as{l}.
\end{equation*}

(iv) By Theorem \ref{hitretrelcor} also
\begin{equation*}
\mu\times\nu)(Q'_l) \Phi_{Q_l'}\xRightarrow{{\mu\times \nu}_{Q_l'}}\Phi_{Exp} \;\;\;\as{l}.
\end{equation*}
At the same time, taking a subsequence if necessary, there are $[0,\infty]$-valued processes $\Phi$ and $\tilde{\Phi}$ such that 
\begin{align*}
&(\mu\times\nu)(Q_l)\Phi_{Q_l}\xRightarrow{\mu\times \nu}\Phi\;\;\;\as{l},\; \textit{ and}\\
&(\mu\times\nu)(Q_l)\Phi_{Q_l}\xRightarrow{(\mu\times \nu)_{Q_l}}\tilde{\Phi}\;\;\;\as{l}.
\end{align*}

Hence 
\begin{equation*}
D\left(\Phi_{Exp},\tilde{\Phi}\right)\leq 7\epsilon,
\end{equation*}
where $D$ is given by Lemma \ref{weakapproxlem}. Since this is true for every $\epsilon>0$ we have $\tilde{\Phi}\overset{d}{=}\Phi_{Exp}$, and by Theorem \ref{hitretrelcor} also $\Phi\overset{d}{=}\Phi_{Exp}$.
\end{proof}

\section{The skewing time}
Here we will prove Theorem \ref{plt:2}, to do this we will verify that the map $T(x,y)=G_{\tau(y)}(x)$ satisfies superpolynomial mixing of all orders, i.e condition (SPM) of Remark \ref{delaykapretrem}(iii).

\begin{lemma}\label{skewingtimelemma}
Under the assumptions of Theorem \ref{plt:2}, suppose that $\sum_{l\geq 1} r_l^{\frac{1}{2}(d'-\delta_2 \kappa)} <\infty$. Then, for each $t>0$, there is a set $\mathfrak{G}_t$ with $\nu(\mathfrak{G}_t)=1$ such that, for $y\in \mathfrak{G}_t$, there are $L_{y,t}>0$ and sets $\mathcal{B}_{l,y,t}\subset \{1,...,\fceil{t}{\mu(A_l)}\}$ with $\# \mathcal{B}_{l,y,t}=o(\mu(A_l)^{-1})$ such that
$$
|\tau_{\tilde{\alpha}_l^{(n)}}(y)-\tau_{\tilde{\alpha}_l^{(m)}}(y)|\geq 
\zeta\left(\tilde{\alpha}_l^{(n)}(y)-\tilde{\alpha}_l^{(m)}(y)\right) 
$$$$
\forall l\geq L_{t,y}, 1\leq n< m \leq \fceil{t}{\mu(A_l)}, n\not\in \mathcal{B}_{l,y,t}
$$
where $\alpha_l=\Phi_{B_l}$.
\end{lemma}
\begin{proof}
Fix $t>0$, to keep notation simple we assume $\mu(A_l)=r_l^d+o(r_l^d)$ and $\nu(B_l)=r_l^{d'}+o(r_l^{d'})$, otherwise there is an extra constant in the estimates below.

(i) We call $\DS n\in \left\{1,...,\fceil{t}{\mu(A_l)}\right\}$ a $(l,y)$-bad return (or simply $(l,y)$-bad) if there is a $m>n$ such that 
\begin{equation*}
|\tau_{\tilde{\alpha}_l^{(n)}}(y)-\tau_{\tilde{\alpha}_l^{(m)}}(y)|< \zeta(\tilde{\alpha}_l^{(n)}(y)-\tilde{\alpha}_l^{(m)}(y)),
\end{equation*}
denote $\mathcal{B}_{l,y}=\{n\geq 1\;|\; n \textit{ is a } (l,y)-\textit{bad return}\}$. Let $\epsilon_1>0$, we call $y\in Y$ an $l$-bad point if $\# \mathcal{B}_{l,y}>r_l^{-d+\epsilon_1}$.

(ii)  Using Corollary \ref{UCcor} and Jegorov's Theorem, for $\epsilon_2>0$, we can find a measurable $G=G_{\epsilon_2}\subset Y$ with $\nu(G^c)<\epsilon_2$ and an $\tilde{L}\geq 1$ depending on $G$ such that
\begin{equation}\label{alphsmalleq}
\tilde{\alpha}_l^{\left(\ffloor{t}{\mu(A_l}\right)}\leq 2tr_l^{-(d+d')} \;\;\;\forall y\in G, l\geq \tilde{L}.
\end{equation}

(iii) For $l\geq 1$ denote $G_l=\{\tilde{\alpha}_l^{\left(\ffloor{t}{\mu(A_l}\right)}\leq 2tr_l^{-(d+d')}\}$, we have
\begin{align*}
\nu(y\in G_l \;|\;  y \textit{ is } l\textit{-bad})&\leq \frac{\int_{G_l} \# \mathcal{B}_{l,y} \d\nu}{r_l^{-d+\epsilon_1}} \leq r_l^{d-\epsilon_1} \sum_{n=1}^{\fceil{t}{\mu(A_l)}} \nu(y\in G_l\;|\; n \textit{ is } (l,y)\textit{-bad})\\
&\leq r_l^{d-\epsilon_1} \sum_{j=1}^{\lceil 2tr_l^{-(d+d')}\rceil} \nu \left(\exists i\geq 1\;|\; |\tau_{j+\tilde{\alpha}_l^{(i)}}-\tau_j|<\zeta\left(\tilde{\alpha}_l^{(i)}\right)\right) \\
&\leq r_l^{d-\epsilon_1} \sum_{j=1}^{\lceil 2tr_l^{-(d+d')}\rceil} \nu \left(R^{-j}\left(\exists i\geq 1\;|\; |\tau_{\tilde{\alpha}_l^{(i)}}|<\zeta\left(\tilde{\alpha}_l^{(i)}\right)\right)\right) \\
&\leq r_l^{d-\epsilon_1}  \sum_{j=1}^{\lceil 2tr_l^{-(d+d')}\rceil} \nu(\exists i\geq r_l^{-\delta_2} \;|\; |\tau_i|<\zeta(i))\\
&\leq 2Kt r_l^{d-\epsilon_1-d-d'+\delta_2\kappa}\leq 2Kt r_l^{\delta_2\kappa-d'-\epsilon_1},
\end{align*}
for some constant $K>0$, for small enough $\epsilon_1$ this is summable. An application of the Borel-Cantelli Lemma yields that for almost every $y\in Y$; for big enough $l$, either $y\not \in G_l$ or
\begin{equation*}
|\tau_{\tilde{\alpha}_l^{(n)}}(y)-\tau_{\tilde{\alpha}_l^{(m)}}(y)|\geq \zeta(\tilde{\alpha}_l^{(n)}(y)-\tilde{\alpha}_l^{(m)}(y)) \;\;\;\forall 1\leq n< m \leq \fceil{t}{\mu(A_l)}, n\not\in \mathcal{B}_{l,y,}.
\end{equation*}
At the same time, by \eqref{alphsmalleq}, we have $G_l\nearrow Y$. Thus $\nu$-a.e $y\in Y$ is in $G_l$ for big enough $l$, and the conclusion follows.
\end{proof}

\begin{proof}[Proof of Theorem $\ref{plt:2}$]
In order to keep notation simple we will only show the PLT for regularly approximable rectangles, this can be easily extended to geodesic balls, by following the same arguments as in the proof of Theorem \ref{plt:1}. 
\\

(i) For $\nu$-a.e $y\in Y$ and $0=t_0<t_1<...<t_J$ choose $L_y=L_{y,t_1+...+t_J}$ and $\mathcal{B}_{l,y}=\mathcal{B}_{l,y,t_1+...+t_J}$ as in Lemma \ref{skewingtimelemma}. For such a $y$ and $l\geq L_y$ (in the following we suppress $y$ from the notation) consider
\begin{equation*}
S_{t_j,l}=\sum_{i=1}^{\fceil{t_j}{\mu(A_l)}} 1_{A_l}\circ T^{\tilde{\alpha}_l^{(i)}} = S'_{t_j,l}+S_{t_j,l}'',
\end{equation*}
where 
\begin{equation*}
S'_{t_j,l}=\sum_{\substack{i=1,...,\fceil{t_j}{\mu(A_l)} \\ i\not \in \mathcal{B}_{l}}}  1_{A_l}\circ T^{\tilde{\alpha}_l^{(j)}}.
\end{equation*}
As in Proposition \ref{delaykapretprop}, the first goal is to show
\begin{equation*}
\left( S_{t_1,l}-S_{t_0,l},...,S_{t_J,l}-S_{t_{J-1},l} \right) \xRightarrow{\mu} \left( P_{t_1-t_0},...,P_{t_J-t_{J-1}}\right) \;\;\;\as{l},
\end{equation*}
where $(P_t)$ is a standard Poisson process. Since $||S_{t_j,l}''||_{L^1}\rightarrow 0$ for all $j=1,...,J$, it is equivalent to show
\begin{equation*}
\left( S'_{t_1,l}-S'_{t_0,l},...,S'_{t_J,l}-S'_{t_{J-1},l} \right) \xRightarrow{\mu} \left( P_{t_1-t_0},...,P_{t_J-t_{J-1}}\right) \;\;\;\as{l}.
\end{equation*}
For $m_1,...,m_J\geq 1$ it will be enough to show
\begin{equation}\label{momentsinproofofplt:2}
\int_X \prod_{j=1}^J { {S'_{t_j,l}-S'_{t_{j-1},l}} \choose m_j} \d\mu = \prod_{j=1}^J \frac{(t_j-t_{j-1})^{m_j}}{m_j!}.
\end{equation}
We have 
\begin{equation}\label{suminproofofplt:2}
\prod_{j=1}^J { {S'_{t_j,l}-S'_{t_{j-1},l}} \choose m_j} = \sum_{\substack{\fceil{t_{j-1}}{\mu(A_l)}+1\leq k_{j,1}<...<k_{j,m_j} \leq \fceil{t_j}{\mu(A_l)}\\ k_{j,i}\not\in \mathcal{B}_l \textit{ for } j=1,...,J,\; i=1,...,m_j}} \prod_{i,j} 1_{A_l} \circ T^{\tilde{\alpha}_l^{(k_{j,i})}}.
\end{equation}
(ii) Due to assumption (MEM) for $G$, and Lemma \ref{skewingtimelemma}, we have 
\begin{equation*}
\left|\int_X \prod_{j=1}^m f_j \circ T^{\tilde{\alpha}_l^{(n_j)}} \d\mu- \prod_{j=1}^m \int_X f_j \d\mu \right|\leq C_y \psi(\min_{j\not =j'}|\tilde{\alpha}_l^{(n_j)}-\tilde{\alpha}_l^{(n_{j'})}|) \prod_{j=1}^m ||f_j||_{C^r},
\end{equation*}
where $\psi(x)=e^{-\gamma \zeta(x)}$ and $\zeta$ is as in assumption (BA), for $f_1,...,f_m\in C^r$ and $1\leq n_1\leq...\leq n_m\leq \fceil{t1+...+t_J}{r_l^d}$ with $n_i\not\in \mathcal{B}_l$. Due to \eqref{bigreturnsinplt:2} and assumption (BA) we have
\begin{equation*}
\psi(\min_{j\not =j'}|\tilde{\alpha}_l^{(n_j)}-\tilde{\alpha}_l^{(n_{j'})}|)=O(r_l^{w_l}),
\end{equation*} 
for some $w_l>0$ with $w_l\rightarrow\infty$ as $l\rightarrow\infty$. Approximating $1_{A_l}$ by functions in $C^r$ it is straightforward\footnote{The calculation is analogous to Lemma \ref{decorrlem}.} to show that 
\begin{equation*}
\left| \int_X \prod_{i,j} 1_{A_l} \circ T^{\tilde{\alpha}_l^{(k_{j,i})}} \d\mu - \mu(A_l)^{m_1+...+m_J}\right| =o(\mu(A_l)^{m_1+...+m_J}),
\end{equation*}
for $k_{j,i}$ as in \eqref{suminproofofplt:2}. The sum in \eqref{suminproofofplt:2} has 
$$\mu(A_l)^{-(m_1+...+m_J)}\prod_{j=1}^J \frac{(t_j-t_{j-1})^m_j}{m_j!}+o(\mu(A_l)^{-(m_1+...+m_J)})$$ many terms, so \eqref{momentsinproofofplt:2} follows.
\end{proof}

\section{Examples}\label{explesec}
Here we verify conditions (EE) and (BR) for the examples listed in Section~\ref{ScEx}.

\subsection{Diophantine rotations}
\label{SSRot}
Let $\alpha\in \left((0,1)\setminus \mathbb{Q}\right)^{d'}$ satisfy a diophantine condition, i.e there are $C>0$ and $n\geq 1$ such that
\begin{equation*}\tag{D}
\left| \langle k, \alpha \rangle -l\right| >C|k|^{-n} \;\;\;\forall k\in \mathbb{Z}^{d'},k\not=0,l\in\mathbb{Z},
\end{equation*}
and $R=R_{\alpha}:x\mapsto x+\alpha \;(mod\;1)$, for $x\in \mathbb{T}^{d'}$, the rotation by $\alpha$. In fact almost all $\alpha$ satisfy $(D)$ for some $n>d'$ (this is a consequence of a higher dimensional version of Khinchin's Theorem, see eg \cite{DiophantineBRV}).\\
Note that $(D)$ implies that there is a   constant $C'>0$. such that
\begin{equation*}
\left|1-e^{2\pi i \langle k,\alpha\rangle}\right| \geq C' |k|^{-n} \;\;\;\forall k\in \mathbb{Z}^{d'}\setminus \{0\}.
\end{equation*}

Property $(SLR(y^*))$ follows directly from $(D)$. Due to the self-symmetry of $R_{\alpha}$ it is enough to consider returns of $x=0$ to a rectangle $(-r,r)^{d'}$, but
\begin{equation*}
R_{\alpha}^m(0)\in (-r,r)^{d'} \iff |m\alpha - k|<r \;\;\; \textit{for some } k\in \mathbb{Z}^{d'}.
\end{equation*}
Then, by $(D)$, $Cm^{-n}<r$, equivalently $m>(C^{-1}r)^{-\frac{1}{n}}$. Hence \eqref{ballretasm:slr} is satisfied with $\psi(r)=(C^{-1}r)^{-\frac{1}{n}}$.
\\
To show effective equidistribution $(EE)$, we solve the homological equation. Let $f\in H^n(\mathbb{T}^{d'})$ with $\int f=0$. Then 
\begin{equation*}
f(x)=\sum_{k\in \mathbb{Z}^{d'}} a_k e^{2\pi i \langle k,x\rangle},
\end{equation*}
where $\DS \sum_{k\in \mathbb{Z}^{d'}} |a_k|^2 \sum_{j_1+...+j_{d'}=n} \prod_{i=1}^{d'} |k_i|^{2j_i} < \infty$ and $a_0=0$. To solve $f\!\!=\!\!g-g\circ R_{\alpha}$ we write 
\begin{equation*}
g(x)=\sum_{k\in \mathbb{Z}^{d'}} b_k e^{2\pi i \langle k,x\rangle}.
\end{equation*}
By comparing coefficients, this is satisfied for $b_k=\frac{a_k}{1-e^{2\pi i \langle k,\alpha \rangle}}$ for $k\not=0$ and $b_0=0$. We have
$$
\sum_{k\in \mathbb{Z}^{d'}} |b_k|^2  \leq (C')^2 \sum_{k\in \mathbb{Z}^{d'}} |a_k|^2 |k|^{2n} 
 \leq (C')^2 \sum_{k\in \mathbb{Z}} |a_k|^2 \sum_{j_1+...+j_{d'}=n} \prod_{i=1}^{d'} |k_i|^{2j_i} <\infty.
$$
In particular $||g||_{L^2}\leq C'||f||_{H^n}$. \\
Thus for every\footnote{If $\int h \not=0$ consider $f=h-\int h$.} $h\in H^n(\mathbb{T}^{d'})$ we have 
\begin{equation*}
\left| \left| \sum_{j=1}^J h\circ R_{\alpha}^j - J\int_{\mathbb{T}^{d'}} h \d\lambda^{d'} \right| \right|_{L^2} \leq C'' ||h||_{H^n},
\end{equation*}
where $\lambda^{d'}$ is the $d'$-dimensional Lebesgue measure on $\mathbb{T}^{d'}$.\\
Due to Remark \ref{conditionsrem}(iii), condition (EE') is satisfied with $r=n$ and $\delta=0$.\\

We can apply Theorem \ref{plt:1}  with
\begin{equation*}
d>3d'(n+1).
\end{equation*}

\subsection{Horocycle flows}
\label{SSHoro}
Consider the classical horocycle flow $h_t$ on compact homogeneous space $\Gamma \backslash PSL(2,\mathbb{R})$ generated by \begin{equation*}
h_t=\begin{bmatrix}
1&t\\0&1
\end{bmatrix}.
\end{equation*}
For fixed $t>0$ we will consider the time $t$ map $R=h_t$.\\
\\

Condition $(SLR(y^*))$ follows from the relation 
$h_{e^{2s} t}=g_s\circ h_t\circ g_{-s}$, where $g_s$ is the geodesic flow
\begin{equation*}
g_s=\begin{bmatrix}
e^s&0\\
0&e^{-s}
\end{bmatrix}.
\end{equation*}
Indeed, we will show that there is a constant $c>0$ such that for small enough $r>0$, $0<|s|<cr^{-1}$, and $y,y^*\in \Gamma / PSL(2,\mathbb{R})$ with $d(y,y^*)<\frac{c}{2}r^{\frac{1}{2}}$ we have
\begin{equation*}
d(h_s y,y^*)\geq \frac{c}{2}r^{\frac{1}{2}}.
\end{equation*}
By the triangle inequality it is enough to show $d(h_s y,y)\geq cr^{\frac{1}{2}}$. By compactness choose $c=\inf_x d(h_1 x, x)>0$ (by Hedlund's Theorem there are no periodic orbits). Let $t=\frac{\log(|s|)}{2}$, then
\begin{equation*}
d(h_s y,y)=d(g_t h_{sgn(s)} g_{-t} y, g_tg_{-t}y)\geq e^{-|t|} d( h_{sgn(s)} g_{-t} y, g_{-t}y)\geq cr^{\frac{1}{2}}.
\end{equation*}
For small enough $r$, $g_t$ contracts distances at most by a factor of $e^{-|t|}$. Renaming $r=\frac{c}{2}r^{\frac{1}{2}}$ we obtain \eqref{ballretasm:slr} with $\psi(r)=2c^3r^{-2}$.

In order to show effective ergodicity (EE'), we combine 
\cite[Corollary 2.8]{FFT} and \cite[Theorem 1.5]{FF} to conclude that there is a constant $C>0$ with
\begin{equation}\label{horocompee1}
\left|\left|\sum_{j=0}^{n-1} f ( R^j(y)) -n\int f \d\nu\right|\right| \leq C||f||_{W^{15}} N^{\frac{5}{6}+\epsilon} \;\;\;\forall f\in W^s, y\in \Gamma / PSL(2,\mathbb{R}), N\geq 1,
\end{equation}
for all $\epsilon>0$. \\
\\
Indeed, for $s>3$, \cite[Theorem 1.5]{FF} yields 
\begin{equation}\label{horocompee2}
\left| \int_0^T f(h_t(y)) \d t - T \int_{\Gamma / PSL(2,\mathbb{R}} \f\d \nu \right| \leq C(s) ||f||_{W^s} T^{\frac{1}{2}} \log(T) \;\;\;\forall f\in W^s, y\in \Gamma / PSL(2,\mathbb{R}), T>0,
\end{equation}
for some constant $C(s)>0$. \\
A consideration involving twisted integrals as in \cite[Corollary 2.8]{FFT} yields,for $s>14$,
\begin{equation}\label{horocompee3}
\begin{aligned}
\left|\sum_{n=0}^{N-1} f(h_n(y)) - \int_0^N f(h_t(y)) \d t\right| &\leq C'(s) ||f||_{W^s} N^{\frac{5}{6}} \log^{\frac{1}{2}}(N)\\
&\forall f\in W^s, y\in \Gamma / PSL(2,\mathbb{R}), N\geq 1,
\end{aligned}
\end{equation}
for some constant $C'(s)>0$. Now \eqref{horocompee2} and \eqref{horocompee3} together imply \eqref{horocompee1}.
\\
\\

Now Theorem \ref{plt:1} and Remark \ref{pointwiseEEplt} apply with
\begin{equation*}
d>3\times 16 \times 6= 288.
\end{equation*}

\subsection{Skew shifts}\label{skewsec}

Let $\alpha\in (0,1)\setminus \mathbb{Q}$ satisfy the diophantine condition $(D)$ for some $n\geq 2$ and $R:\mathbb{T}^2\rightarrow\mathbb{T}^2$ be given by
\begin{equation*}
R(x,y)=(x+\alpha,y+x).
\end{equation*}
Since $R$ has a Diophantine rotation as a factor $(SLR(y^*))$ is satisfied by \S \ref{SSRot}.

For $k=(k_1,k_2) \in \mathbb{Z}^2$ denote $e_k(x)=e^{2\pi i \langle k,x\rangle}$. 
Note that $$\langle e_k,e_{k'}\circ R^j \rangle_{L^2(\mathbb{T}^2)} =\delta^{(k_1,k_2)}_{(k'_1+jk'_2,k'_2)}.$$ For $f\in H^2(\mathbb{T})$ we can write $f=\sum_{k\in \mathbb{Z}^2} a_k e_k$. If $a_{(k_1,0)}\equiv 0$ (in particular $\int f=0$) then
$$
\sum_{j\geq 1} \left| \langle f,f\circ R^j \rangle_{L^2}\right| = \sum_{j\geq 1} \left|
\sum_{k\in \mathbb{Z}^2} a_{(k_1,k_2)} \overline{a_{(k_1+jk_2,k_2)}}\right|\!
\leq \!\left(\sum_{k\in \mathbb{Z}^2} |a_{(k_1,k_2)}| \right)^2 \!\!\!
\leq C||f||_{H^2}^2,
$$
where $C>0$ does not depend on $f$. From this we obtain
\begin{align*}
\left|\left| \sum_{j=1}^J f \circ R^j\right|\right|_{L^2(\mathbb{T}^2)}^2&
= \sum_{j=1}^J \sum_{j'=1}^J \langle f\circ R^{j'},f\circ R^j \rangle_{L^2(\mathbb{T}^2)}
\\&
\leq \sum_{j=0}^{J-1} (J-j) \langle f,f\circ R^j \rangle_{L^2(\mathbb{T}^2)}
\leq CJ||f||_{H^2(\mathbb{T}^2)}^2.
\end{align*}
For general $f\in H^n(\mathbb{T}^2)$ with $\int f=0$, again write $f=\sum_{k\in \mathbb{Z}^2} a_k e_k$ and set
\begin{align*}
f_1=\sum_{k\in \mathbb{Z}^2,k_2\not =0} a_k e_k \;\;\textit{ and }\;\;f_2=\sum_{k\in \mathbb{Z}^2,k_2=0} a_k e_k.
\end{align*}
Applying the above, and the analysis for diophantine rotations, we find
\begin{equation*}
\left|\left| \sum_{j=1}^J f \circ R^j\right|\right|_{L^2(\mathbb{T}^2)}^2\leq (CJ+C')||f||_{H^n(\mathbb{T}^2)}^2.
\end{equation*}
Thus condition (EE) is satisfied with $\delta=\frac{1}{2}$. \\

So we can apply Theorem \ref{plt:1}  with
\begin{equation*}
d>12(n+1).
\end{equation*}

\subsection{Example \ref{explethm:2}}
Recall the definition of the Weyl Chamber flow on $ \Gamma \backslash SL(d,\mathbb{R})$.
Let $d\geq 3$, and $\Gamma$ be a uniform lattice. Denote by $D_+$ the subgroup of diagonal elements of 
$SL(d,\mathbb{R})$ with positive entries. It is easy to see that $D_+$ is isomorphic to $\mathbb{R}^{d-1}$. 
$D_+$ acts on $\Gamma \backslash Sl(d,\mathbb{R}) $ by right translation, giving us a $\mathbb{R}^{d-1}$-action. By \cite[Theorem 1.1]{MultMixBEG} the action $G$ satisfies (a $\mathbb{R}^{d-1}$ version of) (MEM).\\
\\
The diophantine rotation $R_{\alpha}$ satisfies (EE) and \eqref{bigreturnsinplt:2} by \S \ref{SSRot}. 
Hence we can apply Theorem \ref{plt:2}.

\subsection{Other systems satisfying (EE)}\label{EESec}
From Example \ref{explethm:1} it might seem like (EE) is a very special property and only a few systems satisfy this. The opposite is true, in fact most classical systems have this property.\\
To convince ourselves of this, let us give some more examples and point out the mechanisms.

\begin{definition}
The system $(Y,R,\nu)$ is called \textit{mixing of order $\alpha$} if, for each $f,g\in C^{r'}$ with 
$\int_Y f \d\nu=\int_Y g \d\nu=0$, we have
\begin{equation}
\label{AlMix}
\left| \int_Y f\circ R^n \cdot g \d\nu \right| <||f||_{C^{r'}} ||g||_{C^{r'}} \alpha(n) \;\;\;\forall n\geq 1.
\end{equation}

We say that $(Y,R,\nu)$ is \textit{polynomially mixing} if it is mixing with rate
$\alpha(n)=O(n^{-\epsilon})$ for some $\epsilon>0$.
\end{definition}

\begin{lemma}\label{mixEElem}
Polynomial mixing implies (EE). More precisely if $(Y,R,\nu)$ is mixing of order $\alpha(n)=O(n^{-\epsilon})$, for some $\epsilon>0$, then, for all $\epsilon'>0$, it satisfies property (EE) with \begin{equation}
\delta=\begin{cases}
&\frac{2-\epsilon}{2} \; \textit{ if } \epsilon<1\\
& \frac{1}{2}+\epsilon'\; \textit{ if } \epsilon=1\\
&\frac{1}{2} \; \textit{ if } \epsilon>1.
\end{cases}
\end{equation}
\end{lemma}
\begin{proof}
For $f\in C^{r'}$ with $\int_Y f \d\nu=0$ we have, for $N\geq 1$,
\begin{align*}
\left|\left| \sum_{n=0}^{N-1} f \circ R^n \right|\right|_{L^2}^2 & \leq \sum_{n_1=0}^{N-1} \sum_{n_2=0}^{N-1} \left| \int_Y f\circ R^{n_1} f\circ R^{n-2} \d\nu\right|\\
&\leq 2N \sum_{n=0}^{N-1} \left| \int_Y f \circ R^n \cdot f \d \nu\right|
\leq K ||f||_{C^{r'}}^2 N^{2\delta}
\end{align*}
for some $K>0$.
\end{proof}

\begin{remark}
In fact, the proof above remains valid if \eqref{AlMix} holds for all $n\leq N$ except for 
a subset of $\{1,...,N\}$ of size $N^{1-\bar\epsilon}$ for some $\bar\epsilon>0$. We call such systems
\em{polynomially weakly mixing}.\footnote{In fact a slight modification of the proof shows that 
if $R_1$ is polynomially mixing and $R_2$ satisfies (EE) then $R_1\times R_2$ 
satisfies (EE).}
\end{remark}

Many classical systems exhibit 
polynomial (or faster) mixing
we list just a few examples referring to \cite[Section 8]{DDKNmix} for a more comprehensive list
\begin{itemize}
\item mixing piecewise expanding interval maps \cite[Theorem 3.1]{BSTV} as well as expanding interval maps with critical points and singularities \cite[Theorem 1.5]{StatPropLM}.
\item uniformly hyperbolic systems \cite[Theorem 3.9]{liveranidoc},
\item some quadratic maps \cite[Theorem 3]{QuadraticYoung},
\item non compact translations on finite volume homogeneous spaces
of semisimple Lie groups without compact factors \cite[\S 2.4.3]{KMmix},
\item time change of horocycle flow \cite[Theorem 3]{futimechange}
\end{itemize}

For parabolic and elliptic systems one can often use a harmonic analytic argument akin to (but more involved than) \S \ref{SSRot} or \ref{skewsec}. Other concrete examples include\footnote{In fact the references below show the stronger pointwise bound from Remark \ref{pointwiseEEplt}.}
\begin{itemize}
\item nilflows \cite[Theorem 1.1]{FFnil},
\item almost every interval exchange transformation \cite[Theorem 7.1]{ietEE}
\item time $1$ map of certain smooth surface flows, this follows from a work in progress by the author, where polynomial weak mixing is shown.
\end{itemize}

\section{Robustness of return times}\label{robustnesssec}

Lastly we mention the proof for the delayed PLT. All off the above proofs can be done using $\Phi_{B_l,\alpha}$ instead of $\Phi_{B_l}$, this shows (with the notation from the proof of Theorem \ref{plt:1})
\begin{equation*}
(\mu\times\nu)\left( \bigcup_{k=1}^K A_l^{(k)}\times B_l^{(k)}\right) \Phi_{\bigcup_{k=1}^K A_l^{(k)}\times B_l^{(k)},\alpha} \xRightarrow{\mu} \Phi_{Exp}.
\end{equation*}
To conclude, we only need a version of the approximation Lemma \ref{weakapproxlem} for delayed return times.\\

Let $(M,d_M)$ be a compact metric space, let $\s{\tilde{\vartheta}}{n}$ be a sequence of Lipschitz functions on $M$ dense in $C(M)$, and denote $\vartheta_n=\frac{\tilde{\vartheta}_n}{||\tilde{\vartheta}_n||_{Lip}}$. The metric
\begin{equation*}
D_{M}(\lambda,\lambda')=\sum_{n\geq 1} 2^{-n} \left|\int_M \vartheta_n \d\lambda'- \int_{M} \vartheta_n \d\lambda\right|,
\end{equation*}
for probability measures $\lambda$ and $\lambda'$, models distributional convergence\footnote{In the sense that $\lambda_n\Rightarrow\lambda$ iff $D_{M}(\lambda,\lambda_n)\rightarrow0$.}.

\begin{lemma}\label{robustnesslem}
Let $(X,\mu,T)$ be a probability-preserving dynamical system, $\alpha$ be a sequence of natural numbers, $\s{A}{l}$ be a sequence of rare events, and $\Phi=(\phi^{(1)},\phi^{(2)},...)$ be a random process in $[0,\infty)$. Assume that, for each $\delta>0$, there is a sequence of rare events $\s{A^{(\delta)}}{l}$ with $A_l^{(\delta)}\subset A_l$ and $\mu_{A_l}(A_l\setminus A_l^{(\delta)})<\delta$ such that
\begin{equation*}
\mu(A_l^{(\delta)})\Phi_{A_l^{(\delta)},\alpha} \xRightarrow{\mu} \Phi \;\;\;;\as{l}, \forall\delta>0.
\end{equation*}
Then 
\begin{equation*}
\mu(A_l)\Phi_{A_l,\alpha} \xRightarrow{\mu} \Phi \;\;\;\as{l}.
\end{equation*}
\end{lemma}

\begin{proof}
Taking a subsequence if necessary, we may assume that there is a $[0,\infty]$-valued random process $\Phi'$ with
\begin{equation*}
\mu(A_l)\Phi_{A_l,\alpha} \xRightarrow{\mu} \Phi' \;\;\;\as{l}.
\end{equation*}
For $s,t\in [0,\infty]$ denote $d_{[0,\infty]}(s,t)=|e^{-s}-e^{-t}|$, where by convention $e^{-\infty}=0$, then $([0,\infty],d_{[0,\infty]})$ is a compact metric space. Also the infinite product $([0,\infty]^{\mathbb{N}},d_{[0,\infty]^{\mathbb{N}}})$ is a compact metric space with $diam([0,\infty]^{\mathbb{N}})=1$, where
\begin{equation*}
d_{[0,\infty]^{\mathbb{N}}}((s_j),(t_j))=\sum_{j\geq 1} 2^{-j} d_{[0,\infty]}(s_j,t_j).
\end{equation*}
We claim that for every $\epsilon>0$ there exist $\delta_0>0$ and an $L\geq 1$ such that
\begin{equation}\label{robustnesseq:1}
D_{[0,\infty]^{\mathbb{N}}}\left(law_{\mu}(\mu(A_l^{(\delta)})\Phi_{A_l^{(\delta)},\alpha}),law_{\mu}(\mu(A_l)\Phi_{A_l,\alpha})\right)<5\epsilon \;\;\;\forall l\geq L.
\end{equation}
Then taking $l\rightarrow\infty$ shows $D_{[0,\infty]^{\mathbb{N}}}(\Phi,\Phi')<5\epsilon$ and the conclusion follows by $\epsilon\rightarrow0$.\\
\\
Let $1>\epsilon>0$. First note that\footnote{For $k\geq 1$ and $s,t\in[0,\infty]$ we have $d_{[0,\infty]}(ks,kt)\geq d_{[0,\infty]}(s,t)\leq |s-t|$.} 
\begin{equation*}
D_{[0,\infty]^{\mathbb{N}}}\left(law_{\mu}(\mu(A_l^{(\delta)})\Phi_{A_l^{(\delta)},\alpha}),law_{\mu}(\mu(A_l)\Phi_{A_l^{(\delta)},\alpha})\right)<\delta,
\end{equation*}
so it is enough to show that there exist $\epsilon>\delta_0>0$ and an $L\geq 1$ such that
\begin{equation}\label{robustnesseq:2}
D_{[0,\infty]^{\mathbb{N}}}\left(law_{\mu}(\mu(A_l)\Phi_{A_l^{(\delta)},\alpha}),law_{\mu}(\mu(A_l)\Phi_{A_l,\alpha})\right)<4\epsilon \;\;\;\forall l\geq L.
\end{equation}
Denote $\Phi_{A_l,\alpha}=(\f{A_l,\alpha}^{(1)},\f{A_l,\alpha}^{(2)},...)$ and $\Phi_{A_l^{(\delta)},\alpha}=(\f{A_l^{(\delta)},\alpha}^{(1)},\f{A_l^{(\delta)},\alpha}^{(2)},...)$. Now choose $J\geq 1$ so big that $\sum_{j\geq J} 2^{-j} <\epsilon$, and $T>0$ such that
\begin{equation*}
\mathbb{P}\left(\sum_{j=1}^J \phi^{(j)} >T \right)<\epsilon.
\end{equation*}
For $\delta=\min \left( \frac{\epsilon}{2}, \frac{\epsilon}{2T} \right)$ choose $L\geq 1$ so big that
\begin{equation*}
\mu\left(\sum_{j=1}^J \mu(A_l^{(\delta)}) \f{A_l^{(\delta)},\alpha}^{(j)} >T \right)<2\epsilon \;\;\;\forall l\geq L.
\end{equation*}
Since $\DS \sum_{j=1}^J \f{A_l^{(\delta)},\alpha}^{(j)}>\sum_{j=1}^J \f{A_l,\alpha}^{(j)}$ and $\mu(A_l^{(\delta)})>(1-\delta)\mu(A_l)>\frac{1}{2}\mu(A_l)$, in particular 
\begin{equation*}
\mu\left(\sum_{j=1}^J \mu(A_l) \f{A_l,\alpha}^{(j)} > 2T \right)<2\epsilon \;\;\;\forall l\geq L.
\end{equation*}
Now, for $j=1,...,J$, we have 
\begin{align*}
\mu & \left(\f{A_l,\alpha}^{(j)}\not= \f{A_l^{(\delta)},\alpha}^{(j)}\right) \leq \mu\left(\f{A_l,\alpha}^{(J)}\not= \f{A_l^{(\delta)},\alpha}^{(J)}\right)\\
&\leq \mu\left(\bigcup_{i=1}^{\ffloor{2T}{\mu(A_l)}} T^{-\tilde{\alpha}^{(i)}} (A_l\setminus A_l^{(\delta)})\right) + \mu\left(\sum_{j=1}^J \mu(A_l) \f{A_l,\alpha}^{(j)} > 2T \right)\\
&\leq \frac{2T}{\mu(A_l)}\mu(A_l)\delta+2\epsilon\leq 3\epsilon.
\end{align*}
Thus
\begin{align*}
&D_{[0,\infty]^{\mathbb{N}}}\left(law_{\mu}(\mu(A_l)\Phi_{A_l^{(\delta)},\alpha}),law_{\mu}(\mu(A_l)\Phi_{A_l,\alpha})\right)\\
&=\sum_{j\geq 1} 2^{-j} \int_X d_{[0,\infty]}(\mu(A_l)\f{A_l,\alpha}^{(j)},\mu(A_l)\f{A_l^{(\delta)},\alpha}^{(j))}) \d\mu 
\leq \sum_{j=1}^J 2^{-j} 3\epsilon + \epsilon \leq 4\epsilon
\end{align*}
proving \eqref{robustnesseq:2}.
\end{proof}

To conclude we give a

\begin{proof}[Proof of Lemma \ref{weakapproxlem}]

Denote $\lambda=\mu\times\nu$ and $M=[0,\infty]$, with $D_{M^{\mathbb{N}}}$ as in the proof of Lemma \ref{robustnesslem} we have
\begin{align*}
D_{M^{\mathbb{N}}}&(law_{\lambda_Q}(\lambda(Q)\Phi_{Q}),law_{\lambda_{Q'}}(\lambda(Q')\Phi_{Q'})\\
&\leq \lambda_Q(Q') D_{M^{\mathbb{N}}}(law_{\lambda_Q}(\lambda(Q)\Phi_{Q}),law_{\lambda_{Q}}(\lambda(Q')\Phi_{Q'}) \\
&\;\;\;+ \lambda_Q(Q \setminus Q') D_{M^{\mathbb{N}}}(law_{\lambda_Q}(\lambda(Q)\Phi_{Q}),law_{\lambda_{Q'}}(\lambda(Q')\Phi_{Q'})\\
&\leq D_{M^{\mathbb{N}}}(law_{\lambda_Q}(\lambda(Q)\Phi_{Q}),law_{\lambda_{Q}}(\lambda(Q')\Phi_{Q'}) + \lambda_Q(Q \setminus Q').
\end{align*}
Furthermore
\begin{align*}
D_{M^{\mathbb{N}}}&(law_{\lambda_Q}(\lambda(Q)\Phi_{Q}),law_{\lambda_{Q}}(\lambda(Q')\Phi_{Q'}) \\
&\leq \sum_{j\geq 0} 2^{-j-1} \int_Q d_{M}(\lambda(Q)\f{Q} \circ S_Q^j,\lambda(Q')\f{Q'} \circ S_{Q'}^j) \d\lambda_Q\\
& \leq \sum_{j\geq 0} 2^{-j-1} \int_Q d_{M}(\lambda(Q)\f{Q} \circ S_Q^j,\lambda(Q)\f{Q'} \circ S_{Q'}^j) \d\lambda_Q + \lambda_Q(Q\setminus Q').
\end{align*}
For each $j\geq 0$ we have
\begin{equation*}
\lambda_Q(\f{Q} \circ S_Q^j \not = \f{Q'} \circ S_{Q'}^j) \leq \lambda_Q\left( \bigcup_{i=0}^j S_Q^i (Q\setminus Q')\right) \leq (j+1)\lambda_Q(Q\setminus Q').
\end{equation*}
The claim follows since $\sum_{j\geq 1} j2^{-j}=2$.

\end{proof}

\newcommand{\etalchar}[1]{$^{#1}$}

\end{document}